\documentclass[12pt]{amsart}
\usepackage[T1]{fontenc}
\usepackage[latin9]{inputenc}
\usepackage{geometry}
\usepackage{mathrsfs}
\usepackage{amstext}
\usepackage{amsthm}
\usepackage{amssymb}
\usepackage{bbm}

\usepackage[svgnames]{xcolor}
\usepackage[bookmarksnumbered=true]{hyperref} 
\hypersetup{
	colorlinks = true,
	linkcolor = Blue,
	anchorcolor = blue,
	citecolor = Green,
	filecolor = blue,
	urlcolor = FireBrick
}

\usepackage{aliascnt}
\usepackage[nameinlink]{cleveref}

\makeatletter
\numberwithin{equation}{section}
\numberwithin{figure}{section}

\theoremstyle{plain}
\newtheorem{theorem}{Theorem}[section]

\theoremstyle{definition}
\newaliascnt{definition}{theorem}
\newtheorem{definition}[definition]{Definition}
\aliascntresetthe{definition}
\crefname{definition}{Definition}{Definitions}

\theoremstyle{definition}
\newaliascnt{question}{theorem}

\aliascntresetthe{question}
\crefname{question}{Question}{Questions}

\theoremstyle{definition}
\newaliascnt{remark}{theorem}
\newtheorem{remark}[remark]{Remark}
\aliascntresetthe{remark}
\crefname{remark}{Remark}{Remarks}

\theoremstyle{plain}
\newaliascnt{corollary}{theorem}
\newtheorem{corollary}[corollary]{Corollary}
\aliascntresetthe{corollary}
\crefname{corollary}{Corollary}{Corollaries}

\theoremstyle{plain}
\newaliascnt{lemma}{theorem}
\newtheorem{lemma}[lemma]{Lemma}
\aliascntresetthe{lemma}
\crefname{lemma}{Lemma}{Lemmas}

\theoremstyle{plain}
\newaliascnt{proposition}{theorem}
\newtheorem{proposition}[proposition]{Proposition}
\aliascntresetthe{proposition}
\crefname{proposition}{Proposition}{Propositions}

\theoremstyle{definition}
\newaliascnt{example}{theorem}

\aliascntresetthe{example}
\crefname{example}{Example}{Examples}

\theoremstyle{definition}
\newaliascnt{assumption}{theorem}

\aliascntresetthe{assumption}
\crefname{assumption}{Assumption}{Assumptions}

\theoremstyle{definition}
\newaliascnt{problem}{theorem}

\aliascntresetthe{problem}
\crefname{problem}{Problem}{Problems}

\theoremstyle{definition}
\newaliascnt{claim}{theorem}

\aliascntresetthe{claim}
\crefname{claim}{Claim}{Claims}

\newtheorem*{definition*}{Definition}
\newtheorem*{proposition*}{Proposition} 
\newtheorem*{remark*}{Remark}
\newtheorem*{example*}{Example}
\newtheorem*{problem*}{Problem}
\newtheorem*{question*}{Question}
\newtheorem*{theorem*}{Theorem}
\newtheorem*{lemma*}{Lemma}
\newtheorem*{corollary*}{Corollary}
\newtheorem*{Acknowledgments*}{Acknowledgments}

\renewenvironment{proof}[1][\proofname]{\medskip \noindent {\bfseries #1. }}{\hfill \qedsymbol\medskip}

\MakeRobust{\ref}
\newcommand{\labeltext}[2]{
\@bsphack
\csname phantomsection\endcsname
\def\@currentlabel{#1}{\label{#2}}
\@esphack
}

\usepackage{scalerel}
\usepackage[usestackEOL]{stackengine}
\def\dashint{\,\ThisStyle{\ensurestackMath{%
  \stackinset{c}{.2\LMpt}{c}{.5\LMpt}{\SavedStyle-}{\SavedStyle\phantom{\int}}}%
  \setbox0=\hbox{$\SavedStyle\int\,$}\kern-\wd0}\int}

\usepackage{enumitem}

\DeclareRobustCommand{\SkipTocEntry}[5]{}

\newcommand{\mR}{\mathbb{R}}   
\newcommand{\mC}{\mathbb{C}}   

\newcommand{\abs}[1]{\lvert #1 \rvert}  
\newcommand{\norm}[1]{\lVert #1 \rVert}  

\newcommand{\ol}[1]{\overline{#1}}

\newcommand{\mT}{\mathcal{T}}
\newcommand{\mH}{\mathcal{H}}

\newcommand{\mL}{\mathcal{L}}

\newcommand{\calS}{\mathcal{S}}

\newcommand{\calC}{\mathcal{C}}

\newcommand{\bfa}{\mathbf{a}}
\newcommand{\bfb}{\mathbf{b}}
\newcommand{\bfc}{\mathbf{c}}
\newcommand{\bfd}{\mathbf{d}}

\newcommand{\bfu}{\mathbf{u}}

\newcommand{\bfi}{\mathbf{i}}

\newcommand{\id}{\mathrm{Id}}

\newcommand{\p}{\partial}

\DeclareMathOperator{\supp}{supp}

\DeclareMathOperator{\curl}{curl}

\newcommand{\rmd}{\mathrm{d}}

\makeatother

\usepackage{babel}

\usepackage{orcidlink}

\begin{document}

\title{On elastic scattering behavior of corner domains} 

\begin{sloppypar}

\begin{abstract}
We investigate the scattering of elastic waves by anisotropic inhomogeneous media at a fixed wave number. We prove that, in two dimensions, every penetrable obstacle with piecewise $C^{1}$ boundary and corner singularities, and, in higher dimensions, every obstacle with edge singularities, scatters every incident wave satisfying suitable assumptions. This work extends our previous result \cite{KSS25AnisotropicII} from scalar equations to systems of equations, a generalization that requires a more delicate blowup analysis and more involved computations. Roughly speaking, our results establish the principle that \emph{corners always scatter} for general anisotropic media with variable coefficients. 
\end{abstract}

\subjclass[2020]{35J47, 35P25, 35R35, 74B05} 
\keywords{free boundary, corners always scatter, harmonic polynomials, blowup solutions} 

\author[Kow]{Pu-Zhao Kow\,\orcidlink{0000-0002-2990-3591}}
\address{Department of Mathematical Sciences, National Chengchi University, Taipei 116, Taiwan}
\email{pzkow@g.nccu.edu.tw} 

\author[Salo]{Mikko Salo\,\orcidlink{0000-0002-3681-6779}}
\address{Department of Mathematics and Statistics, P.O. Box 35 (MaD), FI-40014 University of Jyv\"{a}skyl\"{a}, Finland}
\email{mikko.j.salo@jyu.fi}

\author[Shahgholian]{Henrik Shahgholian\,\orcidlink{0000-0002-1316-7913}}
\address{Department of Mathematics, KTH Royal Institute of Technology, SE-10044 Stockholm, Sweden 
}
\email{henriksh@kth.se}

\maketitle

\tableofcontents{}

\section{Introduction} 

\subsection{Background and preliminaries} \label{subsec_background}
In this paper, we investigate the scattering behavior of elastic waves by a bounded medium with anisotropic inhomogeneities, modeled by a system of partial differential equations.
The material properties of the medium are described by spatially varying anisotropic coefficients. In our previous work \cite{KSS25AnisotropicII}, we studied the corresponding problem for acoustic waves, which are governed by a single partial differential equation. Roughly speaking, our results establish that \emph{corners always scatter incident fields}, provided certain assumptions are satisfied. 
The phenomenon of corner scattering was first investigated in pioneering work \cite{BPS14CornerScattering} and has since been extensively studied. Rather than attempting to provide an exhaustive account of the literature, we refer the reader to the recent survey \cite{CV26Survey} for an overview of scattering and non-scattering phenomena for the linear Helmholtz equation with inhomogeneities of nontrivial contrast, as well as to the recent work \cite{RX2026SplitCorners}. 

The primary goal of this paper is to extend some of these results to the elastic setting. We begin by reviewing several related works on elastic wave scattering \cite{BL10ITEP,DLS21ITPElastic,DFL26ITEP}, all of which are restricted to the isotropic constant-coefficient case. Their restriction to constant coefficients stems from the use of special solutions, known as complex geometric optics (CGO) solutions. To the best of our knowledge, no `corners always scatter' result is currently available for variable elastic coefficients, even in the isotropic setting. The present paper considers the more general case of anisotropic variable coefficients. 

To highlight the novelty of the present work, we briefly discuss some of the challenges that arise in the analysis of general elliptic systems: 
\begin{itemize}
\item The unique continuation property, which plays a fundamental role in establishing uniqueness in scattering theory, is known to fail for general elliptic systems, see \cite{KW16counterexampleUCP}. To the best of our knowledge, whether the unique continuation property holds for general anisotropic elastic systems with rough coefficients remains an open problem.
\item Another major challenge is that free boundary techniques for strongly coupled systems remain relatively underdeveloped. For example, an analogue of \cite[Theorem~1.1]{KSS23Anisotropic} for elastic waves is still unknown. The proof of \cite[Theorem~1.1]{KSS23Anisotropic} relies crucially on a mean value theorem for scalar elliptic equations with anisotropic variable coefficients. To the best of our knowledge, no analogous result is currently available for general elliptic systems with anisotropic variable coefficients. 
\end{itemize} 
We discuss additional difficulties in \Cref{rem:other-difficulties} below.

Before introducing our mathematical setting, we first establish some notation. To express the elasticity system in a compact vector form, we adopt the following tensor notation. 
\begin{enumerate}
\item For two vectors (i.e., first-order tensors) $u=(u_{i})$ and $v=(v_{i})$, we denote the scalar $u\cdot v=\sum_{i=1}^{d}u_{i}v_{i}$ and the matrix (i.e., second-order tensor) $(u\otimes v)_{ij}=u_{i}v_{j}$. Similarly, for any vector-valued function $u(x)=(u_{i}(x))$, we denote the scalar-valued function $\nabla\cdot u=\sum_{i=1}^{d}\partial_{i}u_{i}$ and the matrix-valued function $(\nabla\otimes u)_{ij} = \partial_{i}u_{j}$. 
\item For a vector $u=(u_{i})$ and a matrix $A=(A_{ij})$, we denote the vector $(u\cdot A)_{i}=\sum_{j=1}^{d}u_{j}A_{ij}$. Similarly, for a matrix-valued function $A(x)=(A_{ij}(x))$, we denote the vector-valued function $(\nabla\cdot A)_{i} = \sum_{j=1}^{d}\partial_{j}A_{ij}$. 
\item For a fourth-order tensor $\mathcal{A}=(A_{ijk\ell})$ and a matrix $\mathcal{B}=(B_{ij})$, the matrices $\mathcal{A}:\mathcal{B}$ and $\mathcal{B}:\mathcal{A}$ are defined by $(\mathcal{A}:\mathcal{B})_{ij}=\sum_{k,\ell=1}^{d}A_{ijk\ell}B_{k\ell}$ and  $(\mathcal{B}:\mathcal{A})_{k\ell}=\sum_{i,j=1}^{d}B_{ij}A_{ijk\ell}$, respectively. 
\item For two fourth-order tensors $\mathcal{A}=(A_{ijk\ell})$ and $\mathcal{B}=(B_{ijk\ell})$, we denote $(\mathcal{A}:\mathcal{B})_{ijk\ell} = \sum_{p,q=1}^{d} A_{ijpq}B_{pqk\ell}$. 
\item For two matrices $A=(A_{ij})$ and $B=(B_{ij})$, we denote $A:B = \sum_{i,j=1}^{d}A_{ij}B_{ij}$. 
\item For any matrix $A=(A_{ij})$, we denote $\abs{A}^{2} := A : \overline{A}$. 
\end{enumerate}

We now present the mathematical formulation of the problem. 
Let $D\subset\mR^{d}$ be a bounded Lipschitz domain with $d\ge 2$ such that $\mR^{d}\setminus\overline{D}$ is connected. Within this domain, let $\rho\in L^{\infty}(D)$ be a positive real-valued scalar function. We extend $\rho$ to $\mR^{d}$ by setting $\rho=1$ outside $D$. 
Let $\calC=(C_{ijk\ell})$ be a real-valued elasticity tensor satisfying the symmetry properties 
\begin{equation}
C_{ijk\ell}(x)=C_{k\ell ij}(x) ,\quad C_{ijk\ell}(x)=C_{jik\ell}(x) \quad \text{in $\mR^{d}$} \label{eq:symmetric-4-tensor}
\end{equation}
for all $1\le i,j,k,\ell \le d$. 
Moreover, as in \cite{Cha02elasticITP}, we assume that the very strong ellipticity condition (also known as Legendre condition \cite[(3.16), Definition~3.36]{GM12EllipticSystems})\footnote{Not to be confused with the strong ellipticity condition, also known as the Legendre-Hadamard condition, see \cite[(3.17), Definition~3.36]{GM12EllipticSystems}.
} holds, that is, there exists a constant $c>0$ such that 
\begin{equation}
c\abs{A}^{2} \le A:\calC(x) : \overline{A} \le c^{-1}\abs{A}^{2} \quad \text{for all $x\in\mR^{d}$} \label{eq:positivity}
\end{equation}
and for all real symmetric matrices $A$. For any vector-valued function $u$, we denote 
\begin{equation*}
\mL^{\calC}u = \nabla\cdot(\calC(x):(\nabla\otimes u)). 
\end{equation*}
According to our notations, $\mL^{\calC}u$ is itself a vector-valued function given by 
\begin{equation}
(\mL^{\calC}u)_{i}=\sum_{j,k,\ell=1}^{d}\partial_{j}(C_{ijk\ell}(x)\partial_{k}u_{\ell}) \quad \text{in $\mR^{d}$}\label{eq:elastic-operator}
\end{equation} 
In addition, we assume that $C_{ijk\ell}|_{\overline{D}} \in C^{\infty}(\overline{D})$ and that $\calC|_{\mR^{d}\setminus\overline{D}}=\calC_0$ is constant, in the sense that 
\begin{equation}
C_{ijk\ell}(x) = \left\{\begin{aligned}
& \tilde{C}_{ijk\ell}(x) && \text{in $\overline{D}$,} \\ 
& (C_0)_{ijk\ell} && \text{in $\mR^{d}\setminus\overline{D}$,}
\end{aligned}\right. \label{eq:elastic-tensor}
\end{equation}
for some $\tilde{\calC}=(\tilde{C}_{ijk\ell})$ with $\tilde{C}_{ijk\ell}\in C^{\infty}(\mR^{d})$.

We illuminate the elastic medium $(D,\calC,\rho)$ with an incident elastic field $u^{\rm inc}$ having a fixed frequency $\kappa>0$ that satisfies 
\begin{equation}
(\mL^{\calC_0} + \kappa^{2})u^{\rm inc} = 0 \quad \text{in $\mR^{d}$.} 
\end{equation}
We assume that $\calC_0$ is isotropic, that is, there exist Lam{\'e} constants $\lambda$ and $\mu$ with $\mu>0$ and $\lambda+2\mu>0$, so \eqref{eq:positivity} is satisfied, and 
\begin{equation}
C_{ijk\ell}(x) \equiv (C_0)_{ijk\ell} = \lambda \delta_{ij} \delta_{k\ell} + \mu(\delta_{ik}\delta_{j\ell} + \delta_{i\ell}\delta_{jk}) \quad \text{in $\mR^{d}\setminus\overline{D}$,} \label{eq:isotropic}
\end{equation}
where $\delta_{ij}$ denotes the Kronecker delta. 
Substituting \eqref{eq:isotropic} into \eqref{eq:elastic-operator} yields 
\begin{equation*}
\mL^{\calC_{0}}u = \mu\Delta u + (\lambda+\mu)\nabla(\nabla\cdot u) \quad \text{in $\mR^{d}\setminus\overline{D}$.} 
\end{equation*} 
Similar to the acoustic scattering theory (see, e.g., \cite{CCH23InverseScatteringTransmission,CK19scattering,KG08Factorization}), the elastic scattering theory \cite{Cha02elasticITP} (see also the classical monograph \cite{KGBB79elastic} and related works \cite{BP08elastic,CGK02LSM,CKAGK07Factorization,Haehner1998,KW21CharacterizeNonradiating}) establishes the existence of a unique scattered elastic field $u^{\rm sc} \in (H_{\rm loc}^{1}(\mR^{d}))^{d}$ which satisfies the outgoing Kupradze radiation condition (\Cref{def:Kupradze}). 
The total field $u^{\rm to} = u^{\rm sc} + u^{\rm inc}$ satisfies 
\begin{equation*}
(\mL^{\calC} + \kappa^{2}\rho(x))u^{\rm to} = 0 \quad \text{in $\mR^{d}$.} 
\end{equation*}

Since $\calC$ is $L^{\infty}$ but may have a jump discontinuity across $\p D$, the solution $u^{\rm to}$ is not in general $H^2$ but the difference of the conormal derivatives from interior and exterior is zero. It is now natural to consider the following definition. 

\begin{definition}
We say that the elastic medium $(D,\calC,\rho)$ is \emph{nonscattering} with respect to incident field $u^{\rm inc}$ if $u^{\rm sc}=0$ in $\mR^{d}\setminus\overline{D}$. 
In this case, the scattered field $u^{\rm sc} \in (H_{\rm loc}^{1}(\mR^{d}))^{d}$ satisfies 
\begin{equation}
\left\{\begin{aligned}
& (\mL^{\tilde{\calC}}+\kappa^{2}\rho(x))u^{\rm sc} = -(\mL^{\tilde{\calC}-\calC_0} + \kappa^{2}(\rho(x)-\id))u^{\rm inc} && \text{in $D$,} \\ 
& u^{\rm sc} = 0 ,\quad \mT^{\tilde{\calC}} u^{\rm sc} = \mT^{\tilde{\calC}-\calC_0} u^{\rm inc} && \text{on $\partial D$,} 
\end{aligned}\right. \label{eq:nonscattering}
\end{equation}
where $\nu$ is the inward unit normal vector to $\partial D$ and the \emph{traction operator} $\mT^{\calC}$ is given by 
\begin{equation*}
\mT^{\calC} u := \nu\cdot(\calC:(\nabla\otimes u)). 
\end{equation*} 
\end{definition}

Note that \eqref{eq:nonscattering} is equivalent to the following elastic transmission eigenvalue problem (cf. \cite{CKW21ITPElastic}) 
\begin{equation*}
\left\{\begin{aligned}
& (\mL^{\tilde{\calC}}+\kappa^{2}\rho(x))u^{\rm to} = 0 ,\quad (\mL^{\calC_{0}}+\kappa^{2})u^{\rm inc}=0 && \text{in $D$,} \\ 
& u^{\rm to} = u^{\rm inc} ,\quad \mT^{\tilde{\calC}} u^{\rm to} = \mT^{\calC_0} u^{\rm inc} && \text{on $\partial D$.} 
\end{aligned}\right. 
\end{equation*}
It should be noted that, in the isotropic constant-coefficient case, the vanishing behavior of elastic transmission eigenfunctions near corners was studied in \cite[Theorem~1.5]{BL10ITEP} and further developed in \cite[Theorems~2.1, 2.3, 3.1, and 3.4]{DLS21ITPElastic}. By formally interpreting corners as points with infinite curvature, \cite[Theorem~4.6]{DFL26ITEP} established an exact quantitative characterization of this phenomenon. The results in \cite{BL10ITEP,DLS21ITPElastic,DFL26ITEP} are based on the construction of special solutions known as complex geometric optics (CGO) solutions.

Similar to \cite{KSS23Anisotropic,KSS25AnisotropicII}, the elastic nonscattering problem can be reformulated as a Bernoulli-type free boundary problem of the form
\begin{equation*}
(\mL^{\tilde{\calC}}+\kappa^{2}\rho(x))u^{\rm sc} = f \mL^{d}\lfloor D + g \mH^{d-1}\lfloor\partial D ,\quad u^{\rm sc}|_{\mR^{d}\setminus\overline{D}} = 0, 
\end{equation*}
see \eqref{eq:nonscattering-Bernuolli} below. Hereafter, we denote by $\mL^{d}\lfloor D \equiv \chi_{D}$ the $d$-dimensional Lebesgue measure restricted to $D$, and by $\mH^{d-1}\lfloor\Gamma$ the $(d-1)$-dimensional Hausdorff measure restricted to $\Gamma$. The latter can be understood in the distributional sense:  
\begin{equation*}
(g\mH^{d-1}\lfloor\Gamma)(\phi) := \int_{\Gamma} g\phi\,\rmd \mH^{d-1} \quad \text{for all $\phi\in C_{c}^{\infty}(\mR^{d})$.} 
\end{equation*}

\subsection{Vanishing Bernoulli condition} 

Our first theorem states that any piecewise $C^{1}$ planar obstacle whose boundary has a singular (non-$C^{1}$) point always scatter in many cases: 

\begin{theorem}\label{thm:1}
Let $D$ be a bounded open set in $\mR^{2}$, and suppose that $D$ is simply connected and has piecewise $C^{1}$ boundary (cf. \cite[Definition~1.2]{KSS25AnisotropicII}). 
Let $\kappa>0$, let $\rho\in L^{\infty}(D)$ be a positive real-valued function, and let $\calC$ be the elastic tensor in \eqref{eq:elastic-tensor} satisfying the assumptions stated above. Suppose that there exists $x_0\in\partial D$ such that the contrast $h(x):= \kappa^{2}\left(\rho(x)-1\right)\chi_{D}$ satisfies the non-degeneracy condition at $x_{0}\in\partial D$: 
\begin{equation}
\text{there is $r>0$ such that $h\in C^{\alpha}(\overline{D})\cap B_{r}(x_0)$ and $h(x_0)\neq 0$,} \label{eq:non-degeneracy-condition}
\end{equation}
and the elastic tensor $\tilde{\calC}$ satisfies 
\begin{equation}
\abs{\tilde{\calC}(x)-\calC_0} \le C\abs{x-x_0}^{2+\alpha} ,\quad \max_{ijk\ell} \abs{\nabla \tilde{C}_{ijk\ell}} \le C\abs{x-x_0}^{1+\alpha}. \label{eq:degenerate-condition}
\end{equation}
If $\partial D$ is not $C^{1}$ near $x_0$, then the anisotropic medium $(D,\calC,\rho)$ scatters nontrivially every incident wave $u^{\rm inc}$ satisfying either $u^{\rm inc}(x_{0})\neq 0$ or $\nabla u^{\rm inc}(x_{0})\neq 0$. 
\end{theorem}

\begin{remark*}
Compared with our previous result in \cite[Theorem~1.4]{KSS25AnisotropicII}, we establish the result only for incident waves $u^{\rm inc}$ satisfying either $u^{\rm inc}(x_{0})\neq 0$ or $\nabla u^{\rm inc}(x_{0})\neq 0$. This restriction is due to the lengthy calculations required in the higher-order case, which are difficult to verify, see \Cref{rem:difficulty-higher-order}.
\end{remark*}

\begin{remark}\label{rem:other-difficulties}
Motivated by the previous works \cite{KSS25AnisotropicII,SS25vanishingcontrast}, one may ask whether, in the two-dimensional case, the result of \Cref{thm:1} extends to Lipschitz domains $D$. Unlike the piecewise $C^{1}$ case considered in \Cref{thm:1}, the analysis in this setting appears to require a suitable monotonicity formula. 
We leave this problem for future investigation. 
\end{remark}

The following theorem gives an analogous result for domains in $\mR^{d}$, $d\ge 3$, with edge singularities. 

\begin{theorem}\label{thm:2}
Let $D$ be a bounded open set in $\mR^{d}$ with $d\ge 3$ such that $\mR^{d}\setminus\overline{D}$ is connected. Let $\kappa>0$, let $\rho\in L^{\infty}(D)$ be a positive real-valued function, and let $\calC$ be the elastic tensor in \eqref{eq:elastic-tensor} satisfying the assumptions stated above. Suppose that \eqref{eq:non-degeneracy-condition} and \eqref{eq:degenerate-condition} are satisfied. If $\partial D$ contains an edge point $x_0\in\partial D$ (cf. \cite[Definition~1.3]{KSS25AnisotropicII}), then the anisotropic medium $(D,\calC,\rho)$ scatters nontrivially every incident wave $u^{\rm inc}$ satisfying either $u^{\rm inc}(x_{0})\neq 0$ or $\nabla u^{\rm inc}(x_{0})\neq 0$. 
\end{theorem}

It is worth noting that, in the isotropic constant-coefficient case, an exact quantitative analogue of this result was established in \cite[Theorem~4.3]{DFL26ITEP} by formally interpreting corners as points of infinite curvature.

\subsection{Nonvanishing Bernoulli condition} 

We also obtain several results in the regime where \eqref{eq:degenerate-condition} does not hold. 

\begin{theorem}\label{thm:3} 
Let $D$ be a bounded open set in $\mR^{2}$, and suppose that $D$ is simply connected and has piecewise $C^{1}$ boundary (cf. \cite[Definition~1.2]{KSS25AnisotropicII}). 
Let $\kappa>0$, let $\rho\in L^{\infty}(D)$ be a positive real-valued function, and let $\calC$ be the elastic tensor in \eqref{eq:elastic-tensor} satisfying the assumptions stated above. Suppose that there exists $x_0\in\partial D$ such that the contrast $h(x):= \kappa^{2}\left(\rho(x)-1\right)\chi_{D}$ satisfies the degeneracy condition at $x_{0}\in\partial D$: 
\begin{equation}
\abs{h(x)} \le C\abs{x-x_{0}}^{\alpha} \label{eq:cond1}
\end{equation}
and $\tilde{\calC}-\calC_{0}$ satisfies the asymptotic expansion at $x_{0}\in\partial D$: 
\begin{equation}
\tilde{\calC}(x)-\calC_{0} = \tilde{\calC}(x_{0}) -\calC_{0} + \tilde{R}(x-x_{0}) \label{eq:cond2}
\end{equation}
with $\abs{\tilde{R}}\le C\abs{x}^{\alpha}$ and $\abs{\nabla\otimes\tilde{R}}\le C\abs{x}^{\alpha-1}$. 
If $\partial D$ is not $C^{1}$ near $x_{0}$, then the obstacle $(D,\calC,\rho)$ scatters every incident wave $u^{\rm inc}$ satisfying 
\begin{equation}
\lim_{x\rightarrow x_{0}}\mT^{\tilde{\calC}-\calC_{0}}u^{\rm inc}(x) \not\equiv0. \label{eq:nondegenerate-Bernoulli}
\end{equation}
\end{theorem}

\begin{remark*}
The condition \eqref{eq:nondegenerate-Bernoulli} implies that $\tilde{\calC}(x_{0})\not\equiv \calC_{0}$ and $\nabla\otimes u^{\rm inc}(x_{0})\not\equiv 0$. This situation corresponds to that in \cite[Remark 1.8]{KSS25AnisotropicII} with $m=1$.
\end{remark*}

An analogous result for domains in $\mR^{d}$ with $d\ge 3$ that contain edge singularities is given in the following theorem.  

\begin{theorem}\label{thm:4} 
Let $D$ be a bounded open set in $\mR^{d}$ with $d\ge 3$ such that $\mR^{d}\setminus\overline{D}$ is connected. 
Let $\kappa>0$, let $\rho\in L^{\infty}(D)$ be a positive real-valued function, and let $\calC$ be the elastic tensor in \eqref{eq:elastic-tensor} satisfying the assumptions stated above. Suppose that \eqref{eq:cond1} and \eqref{eq:cond2} are satisfied. 
If $\partial D$ contains an edge point $x_0\in\partial D$ (cf. \cite[Definition~1.3]{KSS25AnisotropicII}), then the obstacle $(D,\calC,\rho)$ scatters every incident wave $u^{\rm inc}$ satisfying \eqref{eq:nondegenerate-Bernoulli}. 
\end{theorem}

\section{Proof of the main results in the vanishing Bernoulli case} 

Before presenting the main proof, we first highlight the following important observation concerning the symmetry of the elastic operator.

\begin{remark}[rotational transformation of $\mL^{\calC_0}$]\label{rem:rotational-symmetry}
Let $u$ be a solution to $\mL^{\calC_0}u=f$ in $\mR^{d}$, that is, 
\begin{equation*}
\sum_{j,k,\ell=1}^{d}(C_0)_{ijk\ell}\partial_{j}\partial_{k}u_{\ell} = f_{i}\text{ in $\mR^{d}$} \quad \text{for all $i=1,\cdots,d$.}
\end{equation*}
Let $A$ be real orthogonal matrix, i.e., $A^{-1}=A^{\intercal}$, and define $u_{A}(x):=Au(A^{\intercal}x)$. Since 
\begin{equation*}
\partial_{j'}\partial_{k'}(u_{A}(x))_{\ell'} = \sum_{\ell''=1}^{d}A_{\ell'\ell''} \partial_{j'}\partial_{k'}\left(u_{\ell''}(A^{\intercal}x)\right) = \sum_{j'',k'',\ell''=1}^{d}A_{\ell'\ell''} A_{j'j''}A_{k'k''}\partial_{j''}\partial_{k''}u_{\ell''}(A^{\intercal}x), 
\end{equation*} 
a direct computation shows that $u_{A}$ satisfies 
\begin{equation*}
\sum_{j',k',\ell'=1}^{d}\left(\sum_{j,k,\ell=1}^{d}(C_0)_{ijk\ell}(A^{\intercal})_{jj'}(A^{\intercal})_{kk'}(A^{\intercal})_{\ell\ell'}\right)\partial_{j'}\partial_{k'}(u_{A})_{\ell'} = f_{i}(A^{\intercal}\cdot)\text{ in $\mR^{d}$} 
\end{equation*}
for all $i=1,\cdots,d$. Equivalently, defining the transformed tensor 
\begin{equation*}
(C_{0,A})_{i'j'k'\ell'} := \sum_{i,j,k,\ell=1}^{d}(C_0)_{ijk\ell}(A^{\intercal})_{ii'}(A^{\intercal})_{jj'}(A^{\intercal})_{kk'}(A^{\intercal})_{\ell\ell'}, 
\end{equation*}
we obtain 
\begin{equation*}
\mL^{\calC_{0,A}}u_A = A f(A^{\intercal}\cdot) \quad \text{in $\mR^{d}$.} 
\end{equation*}
We note that, in general, $\mL^{\calC_{0}}$ is \emph{not rotationally invariant} when $\calC_{0}$ is anisotropic. Nevertheless, it still satisfies the standard assumptions of elasticity. In particular, under our assumption that $\calC_0$ is isotropic, we compute that 
\begin{equation*}
\begin{aligned}
(C_{0,A})_{i'j'k'\ell'} &= \sum_{i,j,k,\ell=1}^{d}\left(\lambda \delta_{ij} \delta_{k\ell} + \mu(\delta_{ik}\delta_{j\ell} + \delta_{i\ell}\delta_{jk})\right)(A^{\intercal})_{ii'}(A^{\intercal})_{jj'}(A^{-1})_{kk'}(A^{\intercal})_{\ell\ell'} \\ 
&= \sum_{i,j,k,\ell=1}^{d}\left(
\begin{aligned} 
&\lambda (A_{i'i}\delta_{ij}(A^{\intercal})_{jj'}) (A_{k'k}\delta_{k\ell}(A^{\intercal})_{\ell\ell'}) \\
&\quad + \mu(A_{i'i}\delta_{ik}(A^{\intercal})_{kk'})(A_{j'j}\delta_{j\ell}(A^{\intercal})_{\ell\ell'}) \\
&\quad + \mu(A_{i'i}\delta_{i\ell}(A^{\intercal})_{\ell\ell'})(A_{j'j}\delta_{jk}(A^{\intercal})_{kk'})
\end{aligned} 
\right) \\ 
&= \lambda \delta_{i'j'} \delta_{k'\ell'} + \mu(\delta_{i'k'}\delta_{j'\ell'} + \delta_{i'\ell'}\delta_{j'k'}) \\ 
&= (C_{0})_{i'j'k'\ell'}, 
\end{aligned}
\end{equation*} 
thus, the isotropic elasticity tensor is invariant under rotations.
\end{remark}

We assume the conditions in \Cref{subsec_background}, except that $\calC_0$ does not need to be isotropic. We begin by showing that the nonscattering condition \eqref{eq:nonscattering}, where we assume that  $u^{\rm sc}=0$ in $\mR^{d}\setminus\overline{D}$, can indeed be rephrased as a Bernoulli free boundary problem.

For each $\psi\in (C_{c}^{\infty}(\mR^{d}))^{d}$, one computes that (we slightly abuse some notations here) 
\begin{equation*}
\begin{aligned}
& \int_{\mR^{d}} \psi\cdot \left( -(\mL^{\calC-\calC_0} + \kappa^{2}(\rho(x)-\id))u^{\rm inc} \right)\mL^{d}\lfloor D \, \rmd x \\ 
& \quad = \int_{D} \psi\cdot \left( -(\mL^{\calC-\calC_0} + \kappa^{2}(\rho(x)-\id))u^{\rm inc} \right) \, \rmd x \\ 
& \quad = \int_{D} \psi \cdot \left( (\mL^{\calC}+\kappa^{2}\rho(x))u^{\rm sc} \right) \, \rmd x \\ 
& \quad = \int_{D} \left( - (\nabla\otimes\psi) : \calC(x) : (\nabla\otimes u^{\rm sc}) + \kappa^{2}\rho(x) u^{\rm sc} \right) \, \rmd x \\
& \qquad + \int_{\partial D} \psi\cdot \left( \nu\cdot(\tilde{\calC}:(\nabla\otimes u^{\rm sc})) \right) \, \rmd\mH^{d-1} \\ 
& \quad = \int_{\mR^{d}} \left( - (\nabla\otimes\psi) : \calC(x) : (\nabla\otimes u^{\rm sc}) + \kappa^{2}\rho(x) u^{\rm sc} \right) \, \rmd x \\
& \qquad + \int_{\partial D} \psi\cdot \left( \nu\cdot((\tilde{\calC}-\calC_0):(\nabla\otimes u^{\rm inc})) \right) \, \rmd\mH^{d-1} \\ 
& \quad = \int_{\mR^{d}} \psi \cdot \left( (\mL^{\calC} + \kappa^{2}\rho(x))u^{\rm sc} \right) \, \rmd x \\
& \qquad + \int_{\mR^{d}} \psi\cdot \left( \nu\cdot((\tilde{\calC}-\calC_0):(\nabla\otimes u^{\rm inc})) \right) \mH^{d-1}\lfloor\partial D \, \rmd x. 
\end{aligned}
\end{equation*}
This shows that \eqref{eq:nonscattering}  is equivalent to 
\begin{subequations} \label{eq:nonscattering-Bernuolli}
\begin{equation}
(\mL^{\tilde{\calC}}+\kappa^{2}\rho(x))u^{\rm sc} = f \mL^{d}\lfloor D + g \mH^{d-1}\lfloor\partial D ,\quad u^{\rm sc}|_{\mR^{d}\setminus\overline{D}} = 0 
\end{equation} 
in the sense of distributions, where 
\begin{equation}
f = -(\mL^{\tilde{\calC}-\calC_{0}}+h)u^{\rm inc} ,\quad g = \nu\cdot(\tilde{\calC}-\calC_{0}): (\nabla\otimes u^{\rm inc}) \label{eq:fg-functions}
\end{equation}
\end{subequations}
and $h(x):= \kappa^{2}(\rho(x)-1)\chi_{D}$.

By \cite[Proposition~5.8]{GM12EllipticSystems}, for each open set $\Omega\subset\mR^{d}$ with $d\ge 3$, there exists a constant $c_{*}>0$, depending only on the dimension $d$ and the elliptic constant $c$ in \eqref{eq:positivity}, such that any weak solution $u\in (H_{\rm loc}^{1}(\Omega))^d$ of $\mL^{\tilde{\calC}}u=0$ in $\Omega$ satisfies 
\begin{equation*}
\int_{B_{\rho}(x_0)}\abs{u}^{2}\,\rmd x \le c_{*}\left(\frac{\rho}{R}\right)^{d}\int_{B_{R}(x_0)} \abs{u}^{2}\,\rmd x 
\end{equation*}
for all balls $B_{\rho}(x_0)\subset B_{R}(x_0)\subset\Omega$. In particular, the operator $\mL^{\tilde{\calC}}$ satisfies property~(H) in \cite[Definition~2.1]{HK02fundamentalsolution}. Consequently, for $d\ge3$, one can construct a fundamental matrix $\tilde{\Gamma}$ for the general anisotropic elastic operator $\mL^{\tilde{\calC}}$ as in \cite[Section~3]{HK02fundamentalsolution} satisfying 
\begin{equation*}
\abs{\tilde{\Gamma}(x,y)} \le C\abs{x-y}^{2-d} \quad \text{for all $x\neq y$.} 
\end{equation*} 
One may also refer to \cite{DHM18FundamentalSolution}, however, their approach relies on a Moser-type inequality, which may be difficult to verify for elliptic systems (see \cite{GM12EllipticSystems}). Nevertheless, we mention their work here, as their method is robust and well suited to treating elliptic equations with lower-order perturbations of low regularity. 

In particular, if $\tilde{\calC}$ is constant and isotropic, by direct computation, one can choose constants $a$ and $b$ such that 
\begin{equation*}
\Gamma(x) = \left\{\begin{aligned} 
& a\log\abs{x}\id + b\abs{x}^{-2} x\otimes x && \text{for $d = 2$,} \\ 
& a\abs{x}^{2-d}\id + b\abs{x}^{-d} x\otimes x && \text{for $d\ge 3$,} 
\end{aligned}\right. 
\end{equation*}
is a fundamental matrix. In the special case $d=3$, this $\Gamma$ is known as the Kelvin matrix, see \cite[Section~1, Chapter~II]{KGBB79elastic}. It is also worth noting that Nakamura and Tanuma \cite{GTfund1,GTfund2} derived an explicit formula for the fundamental matrix of $\mL^{\tilde{\calC}}$ for an anisotropic constant tensor $\calC$ in the case $d=3$, expressed in terms of the eigenfunctions of the Stroh's eigenvalue problem.

The subsequent lemma, concerning the H{\"o}lder regularity of solutions, can be proved using essentially the same argument as in \cite[Lemma~2.1]{ACS01FreeBoundaryCalderon}. The only modifications are that one uses the properties of the fundamental matrix $\tilde{\Gamma}$ given as above, and observes that $\norm{u}_{L^{\infty}}$ can be controlled by $\norm{u}_{L^{2}}$ via an interior elliptic regularity estimate, see \cite[Theorem~5.20]{GM12EllipticSystems}\footnote{with the choice $\lambda=0$, the Morrey space $L^{p,\lambda}$ coincides with the usual $L^{p}$ space, which recovers exactly the Calder{\'o}n-type estimate for elliptic systems.}.

\begin{lemma}\label{lem:Hoelder-regularity}
Suppose $u\in (H_{\rm loc}^{1}(\mR^{d}))^{d}$ satisfies 
\begin{equation*}
(\mL^{\tilde{\calC}}+\kappa^{2}\rho(x))u = f \mL^{d}\lfloor D + g \mH^{d-1}\lfloor\partial D 
\end{equation*}
for some $f\in (L^{\infty}(D))^{d}$ and $g\in (L^{\infty}(\partial D))^{d}$. Then $u \in (C_{\rm loc}^{\beta})^{d}$ near $\partial D$ for some $0<\beta<1$.  
\end{lemma}

The following lemma, concerning the Lipschitz regularity of solutions, can be proved using essentially the same argument as in \cite[Lemma~2.2]{KSS23Anisotropic}, which is itself a slight modification of the ideas in \cite[Lemma~2.2]{ACS01FreeBoundaryCalderon}. The only additional ingredient is the analyticity of solutions to elasticity systems with constant coefficients, see \cite{MN57AnalyticSolutionEllipticSystems}. 
In the case where $\calC_0$ is isotropic, this step can alternatively be replaced by the unique continuation principle (UCP) for the isotropic elasticity system \cite{DLW20UCPLame,LNUW11UCPLame,LW05UCPLame}. To the best of our knowledge, the question of whether the unique continuation property holds for general elasticity systems remains open. It is worth noting that unique continuation may fail for general elliptic systems, see \cite{KW16counterexampleUCP}.

\begin{lemma}\label{lem:Lipschitz-regularity} 
Suppose $u\in (H_{\rm loc}^{1}(\mR^{d}))^{d}$ satisfies 
\begin{equation*}
(\mL^{\tilde{\calC}}+\kappa^{2}\rho(x))u = f \mL^{d}\lfloor D + g \mH^{d-1}\lfloor\partial D ,\quad u|_{\mR^{d}\setminus\overline{D}}=0 
\end{equation*}
for some $f\in (L^{\infty}(D))^{d}$ and $g\in (L^{\infty}(\partial D))^{d}$. Then $u\in (C_{\rm loc}^{0,1})^{d}$ near $\partial D$.  
\end{lemma}

Here we note that, following \cite[Lemmas~2.1 and 2.2]{ACS01FreeBoundaryCalderon}, we first prove \Cref{lem:Hoelder-regularity,lem:Lipschitz-regularity} for $d\ge3$. The case $d=2$ then follows by simply ``adding one dimension''.

For convenience and without loss of generality, we assume from now on that $x_0=0 \in \partial D$ and we denote $B_{R}=B_{R}(0)$ for all $R>0$. 
Consequently, following the approach in \cite[Lemma~2.1]{KSS25AnisotropicII}, we can derive the following lemma (the details are omitted).

\begin{lemma}\label{lem:Lipschitz}
Let $D$ be a bounded Lipschitz domain in $\mR^{d}$ with $0\in\partial D$. 
Let $q \in (L^{\infty}(B_{2}))^{d\times d}$, let $m\ge 0$ be an integer, and suppose that $u\in (H^{1}(B_{2}))^{d}$ solves 
\begin{equation}
(\mL^{\tilde{\calC}}+q(x))u = f\mL^{d}\lfloor B_{2} + g\mH^{d-1}\lfloor\partial D \text{ in $B_{2}$} ,\quad u|_{B_{2}\setminus\overline{D}} = 0 \label{eq:PDE-blowup}
\end{equation}
with $\abs{f(x)}\le C\abs{x}^{m}$ a.e. in $B_{2}$ and $\abs{g(x)}\le C\abs{x}^{m+1}$ for $\mH^{d-1}$-a.e. $x\in\partial D\cap B_{\epsilon}$. Then 
\begin{equation}
\abs{u(x)} + \abs{x}\abs{\nabla u(x)} \le C\abs{x}^{m+2} \quad \text{in $B_{1}$.} \label{eq:optimal-decay} 
\end{equation}
\end{lemma}

Now let $u$ be the vector-valued function given in \Cref{lem:Lipschitz}. Note that the function $u_{r}(x):=u(rx)/r^{m+2}$ also satisfies the estimate \eqref{eq:optimal-decay}. In view of the Banach-Alaoglu theorem, we say that $v$ is a \emph{blowup limit of $u$ of order $m+2$ at $0$} if there is a sequence $r_{j}\rightarrow 0$ so that $u_{r_{j}}\rightarrow v$ in $(C^{0,1}(\overline{B_{1}}))^{d}$ weak-$\star$. 

Following \cite[Lemma~2.5]{SS25vanishingcontrast}, we assume that, there exists a nonnegative integer $m$ such that 
\begin{subequations} \label{eq:homogeneous-assumption1}
\begin{equation}
f = P + R 
\end{equation}
where $P\not\equiv0$ and each component $P_{i}$ either vanishes identically or is a nontrivial homogeneous polynomial of degree $m$ and $\abs{R(x)}\le C\abs{x}^{m+\alpha}$ for some $\alpha>0$ and 
\begin{equation}
\abs{g(x)} \le C\abs{x}^{m+1+\alpha} \text{ for $\mH^{d-1}$-a.e. $x\in\partial D\cap B_{2}$.}
\end{equation}
We also assume that 
\begin{equation}
\abs{\tilde{\calC}(x)-\calC_0} \le C\abs{x} ,\quad \max_{ijk\ell} \abs{\nabla \tilde{C}_{ijk\ell}} \le C. 
\end{equation}
\end{subequations} 
The following lemma can be proved using the exact same argument in \cite[Lemma~2.4]{KSS25AnisotropicII}, where the ideas are adapted from \cite[Lemma~4.1]{SS25vanishingcontrast}: 

\begin{lemma}\label{lem:non-degeneracy} 
Suppose that the assumptions of \Cref{lem:Lipschitz} as well as \eqref{eq:homogeneous-assumption1} hold. Then for any $\epsilon\in(0,1)$, there exists a pair of positive numbers $(r_{\epsilon},c_{\epsilon})$ such that 
\begin{equation*}
\norm{u}_{L^{\infty}(B_{\epsilon\abs{x}}(x))} \ge c_{\epsilon}\abs{x}^{m+2} \quad \text{for all $x\in\overline{D}\cap\overline{B_{r_{\epsilon}}}$}. 
\end{equation*} 
\end{lemma}

With \Cref{lem:non-degeneracy} at hand, we can establish the following weak flatness properties for $\partial D$ at $0$, using the exactly same argument in \cite[Lemma~4.3]{SS25vanishingcontrast}: 

\begin{lemma}\label{lem:weak-flatness} 
If $u$ is as in \Cref{lem:Lipschitz} and if $v$ is a blowup limit that satisfies $\supp\,(\abs{v})=\overline{\mR_{+}^{d}}$, then for any $\delta>0$ there exists $r>0$ such that $\partial D\cap B_{r}\subset\left\{\abs{x_{d}}\le\delta r\right\}$. 
\end{lemma}

\subsection{The two-dimensional piecewise \texorpdfstring{$C^{1}$}{C1} obstacle case\label{subsec:2D-case}} 

We begin with the case $d=2$. 
To follow the argument in \cite[Lemma~2.2]{KSS25AnisotropicII}, one needs the homogeneity of the blowup limit, which is obtained via a balanced energy functional and its monotonicity formula. However, for general anisotropic $\tilde{\calC}$, the operator $\mL^{\tilde{\calC}}$ is not rotationally invariant (cf. \Cref{rem:rotational-symmetry} above), making this approach more involved. 
Instead, we follow the approach in \cite[Section~3]{KSS25AnisotropicII}, which avoids the use of a monotonicity functional and instead exploits the relatively simple geometry that $\partial D$ is piecewise $C^{1}$.

Since $\partial D$ is piecewise $C^{1}$, there exists $r_{0}>0$ such that $\partial D\cap B_{r_{0}}=(\Gamma_{-}\cup\Gamma_{+})\cap B_{r_{0}}$, where the $C^{1}$ interfaces $\Gamma_{\pm}$ intersect at $x_{0}=0$. 
By \Cref{rem:rotational-symmetry}, after replacing $\tilde{\calC}$ and $u$ by their rotated versions, we may assume without loss of generality that $\Gamma_{-}\cap B_{r_{0}}\subset\{x_{1}\le 0\}$ and $\Gamma_{+}\cap B_{r_{0}}\subset\{x_{1}\ge 0\}$. Since the normal vector $\nu$ is a continuous vector field on each $\Gamma_{\pm}$, we can define 
\begin{equation*}
\nu_{-} := \lim_{\Gamma_{-}\ni x\rightarrow 0}\nu(x) ,\quad \nu_{+} := \lim_{\Gamma_{+}\ni x\rightarrow 0}\nu(x), 
\end{equation*}
and introduce the straight lines $\tilde{\Gamma}_{-}\subset\{x_{1}\le 0\}$ and $\tilde{\Gamma}_{+}\subset\{x_{1}\ge 0\}$, each perpendicular to $\nu_{-}$ and $\nu_{+}$, respectively. Let $\mathfrak{C}$ be the cone in $\mR^{2}$ such that $\partial\mathfrak{C}=\tilde{\Gamma}_{-}\cup\tilde{\Gamma}_{+}$ and $\mathfrak{C}$ corresponds to the blowup of $D$ at $0$. 
Any blowup limit $v$ of order $m+2$ solves the equation 
\begin{equation}
\mL^{\calC_0}v = H \chi_{\mathfrak{C}} \text{ in $\mR^{2}$} ,\quad v|_{\mR^{2}\setminus\overline{\mathfrak{C}}}=0,  \label{eq:blowup-limit-PDE1}
\end{equation}
where $H \not\equiv 0$ and each component $H_{i}$ vanishes identically or is a nontrivial harmonic homogeneous polynomial of degree $m$. By the interior $L^{p}$-estimates (Calder{\'o}n-Zygmund-type estimates) for solutions of elliptic systems, which may be established via potential theory (see e.g.\ \cite[Proposition A.1]{JulinLiimatainenSalo2017}) or Stampacchia's interpolation theorem \cite{Campanato1966Stampacchia,Stampacchia1965Interpolation} or \cite[Theorem~7.3]{GM12EllipticSystems}, we obtain $v\in (W_{\rm loc}^{2,p}(\mR^{2}))^{2}$ for all $1<p<\infty$. Consequently, $v$ admits a representative in $(C_{\rm loc}^{1,\alpha}(\mR^{2}))^{2}$ for every $0<\alpha<1$. Thus, $v$ satisfies the boundary conditions  
\begin{equation*}
v|_{\partial\mathfrak{C}} = 0 \quad\text{and}\quad \nabla v|_{\partial\mathfrak{C}} = 0. 
\end{equation*}
At this stage, it is not yet known whether $v$ is homogeneous of degree $m+2$. Nevertheless, this property is not needed for the proof in the present setting.

We now assume that $\calC_{0}$ is isotropic, namely, that it is of the form \eqref{eq:isotropic}. 
We first study the case where $\supp(\abs{v})$ is a half space $\{\omega\cdot x > 0\}$ for some fixed $\omega\in\mR^{d}$ with $\abs{\omega}=1$. This case is of particular interest, as we expect $\partial D$ to be regular near the origin. 
In view of the rotational symmetry of $\mL^{\calC_0}$ (cf. \Cref{rem:rotational-symmetry} above), it suffices to consider the choice $\omega=e_{2}$.

We now proceed to prove the following lemma.

\begin{lemma}\label{lem:blowup-wellposedness1}
Let $H\not\equiv0$ be the vector-valued function defined in \eqref{eq:blowup-limit-PDE1}. Consider the system 
\begin{equation}
\left\{\begin{aligned}
& \mL^{\calC_0}v = H \text{ in the half space $\{x_{2}>0\}$,} \\ 
& v|_{\{x_{2}= 0\}} = \partial_{2} v|_{\{x_{2}=0\}} = 0. 
\end{aligned}\right. \label{eq:decoupled-observation}
\end{equation}
This system admits a unique solution $v\in (H_{\rm loc}^{1}(\{x_{2}
>0\}))^{2}$. Moreover, each component of $v$ is a homogeneous polynomial of degree $m+2$. 
\end{lemma}

\begin{proof}[Proof of \Cref{lem:blowup-wellposedness1}] 
Uniqueness follows because any $v \in (H_{\rm loc}^{1}(\{x_{2}>0\}))^{2}$ satisfying $\mL^{\calC_{0}}v=0$ in $\{x_{2}>0\}$ with $v|_{\{x_{2}=0\}} = \mT^{\calC_{0}}v|_{\{x_{2}=0\}} = 0$ can be extended by zero to obtain a solution of 
\begin{equation*}
\mL^{\calC_0}v=0 \text{ in $\mR^{2}$} ,\quad v=0\text{ in $\{x_{2}\le 0\}$.}
\end{equation*}
By analyticity of $v$ \cite{MN57AnalyticSolutionEllipticSystems}, we have $v\equiv 0$.\footnote{If $\calC_0$ is isotropic, one can alternatively apply the unique continuation principle (UCP) for the isotropic elasticity system.}

For the existence part, let $\mathcal{P}_{k}$ denote the space of polynomials of degree at most $k$. The operator 
\begin{equation}
\mL^{\calC_{0}} : (x_{2}^{2}\mathcal{P}_{m})^{2} \rightarrow (\mathcal{P}_{m})^{2} \label{eq:bijective-L}
\end{equation}
is linear and injective by the uniqueness argument above. Since it acts between finite-dimensional spaces of equal dimension, it follows that \eqref{eq:bijective-L} is also surjective, and hence a bijection. This establishes the existence and uniqueness of the solution. 
\end{proof}

Since $\calC_{0}$ is isotropic, \eqref{eq:decoupled-observation} takes the form  
\begin{equation}
\mu \Delta v + (\lambda+\mu)\nabla(\nabla\cdot v) = H \text{ in $\{x_2>0\}$} ,\quad  v|_{\{x_{2}=0\}}=\partial_{x_{2}}v|_{\{x_{2}=0\}} = 0. \label{eq:Lame-2D} 
\end{equation}
Before giving a characterization of $v$, let us first examine the general solution for $m\in\{0,1\}$. 

\begin{lemma}\label{lem:general-solution}
Let $m\in\{0,1\}$. Assume that $H$ admits the $\theta$-Fourier representation 
\begin{equation}
\begin{aligned}
H(re^{\bfi\theta}) &= \left[ \bfa + \frac{m}{2}\frac{\lambda+\mu}{\lambda+3\mu} \left(\begin{array}{cc}1&\bfi\\\bfi&-1\end{array}\right)\bfb \right]r^{m}e^{-\bfi m\theta} \\
&\quad + \left[\bfb + \frac{m}{2}\frac{\lambda+\mu}{\lambda+3\mu} \left(\begin{array}{cc}1&-\bfi\\-\bfi&-1\end{array}\right)\bfa \right]r^{m}e^{\bfi m\theta}
\end{aligned} \label{eq:H}
\end{equation}
for some constant vectors $\bfa,\bfb\in\mC^{2}$, see \Cref{rem:representation-H} below. 
For each open interval  $I\subset(0,2\pi)$, the general solution to 
\begin{equation}
\mu \Delta w + (\lambda+\mu)\nabla(\nabla\cdot w) = H \label{eq:general-solution}
\end{equation}
in the sector $\{ (x_{1},x_{2})\in\mR^{2} : x_{1}+\bfi x_{2}=re^{\bfi\theta} : \theta\in I, r>0 \}$ is given by 
\begin{equation*}
\begin{aligned}
w(re^{\bfi\theta}) &= \bfc r^{m+2}e^{-\bfi(m+2)\theta} -\frac{\lambda+\mu}{2(\lambda +3\mu)}(m+2)\left(\begin{array}{cc}1&-\bfi \\ -\bfi&-1\end{array}\right)\bfc r^{m+2}e^{-\bfi m\theta} \\ 
&\quad + \bfd r^{m+2}e^{\bfi(m+2)\theta} - \frac{\lambda+\mu}{2(\lambda +3\mu)}(m+2)\left(\begin{array}{cc}1&\bfi \\ \bfi&-1\end{array}\right)\bfd r^{m+2}e^{\bfi m\theta} \\ 
&\quad + \frac{1}{2(\lambda+3\mu)(m+1)}\left[ \bfa r^{m+2}e^{-\bfi m\theta} + \bfb r^{m+2}e^{\bfi m\theta} \right] 
\end{aligned}
\end{equation*} 
for constant vectors $\bfc,\bfd\in\mC^{2}$. 
\end{lemma}

We also examine the general solution for $m\ge 2$. 

\begin{lemma}\label{lem:general-solution-higher-order} 
Let $m\ge 2$ be an integer. Assume that $H$ admits the $\theta$-Fourier representation 
\begin{equation}
\begin{aligned}
H(re^{\bfi\theta}) &= \left[ \id - \frac{1}{2}\left(\frac{\lambda+\mu}{\lambda+3\mu}\right)^{2}\left(\begin{array}{cc}1&-\bfi\\\bfi&1\end{array}\right) \right] \bfa r^{m}e^{-\bfi m\theta} \\ 
&\quad + \left[ \id - \frac{1}{2}\left(\frac{\lambda+\mu}{\lambda+3\mu}\right)^{2}\left(\begin{array}{cc}1&\bfi\\-\bfi&1\end{array}\right) \right] \bfb r^{m}e^{\bfi m\theta} 
\end{aligned} \label{eq:H-higher-order}
\end{equation}
for some constant vectors $\bfa,\bfb\in\mC^{2}$, see \Cref{rem:representation-H-higher-order} below. For each open interval  $I\subset(0,2\pi)$, the general solution to \eqref{eq:general-solution} in the sector $\{ (x_{1},x_{2})\in\mR^{2} : x_{1}+\bfi x_{2}=re^{\bfi\theta} : \theta\in I, r>0 \}$ is given by 
\begin{equation*}
\begin{aligned}
w(re^{\bfi\theta}) &= \bfc r^{m+2}e^{-\bfi(m+2)\theta} -\frac{\lambda+\mu}{2(\lambda +3\mu)}(m+2)\left(\begin{array}{cc}1&-\bfi \\ -\bfi&-1\end{array}\right)\bfc r^{m+2}e^{-\bfi m\theta} \\ 
&\quad + \bfd r^{m+2}e^{\bfi(m+2)\theta} - \frac{\lambda+\mu}{2(\lambda +3\mu)}(m+2)\left(\begin{array}{cc}1&\bfi \\ \bfi&-1\end{array}\right)\bfd r^{m+2}e^{\bfi m\theta} \\ 
&\quad + \frac{1}{2(\lambda+3\mu)(m+1)}\bfa r^{m+2}e^{-\bfi m\theta} - \frac{\lambda+\mu}{8(\lambda+3\mu)^{2}}\left(\begin{array}{cc}1&-\bfi\\-\bfi&-1\end{array}\right)\bfa r^{m+2}e^{-\bfi(m-2)\theta} \\ 
&\quad + \frac{1}{2(\lambda+3\mu)(m+1)}\bfb r^{m+2}e^{\bfi m\theta} - \frac{\lambda+\mu}{8(\lambda+3\mu)^{2}}\left(\begin{array}{cc}1&\bfi\\\bfi&-1\end{array}\right)\bfb r^{m+2}e^{\bfi(m-2)\theta}
\end{aligned}
\end{equation*} 
for constant vectors $\bfc,\bfd\in\mC^{2}$. 
\end{lemma}

\begin{remark}\label{rem:representation-H}
Let $m=1$. Given any $\tilde{\bfa},\tilde{\bfb}\in\mC^{2}$, consider the system 
\begin{equation*}
\tilde{\bfa} = \bfa + \frac{1}{2}\frac{\lambda+\mu}{\lambda+3\mu} \left(\begin{array}{cc}1&\bfi\\\bfi&-1\end{array}\right)\bfb ,\quad \tilde{\bfb} = \bfb + \frac{1}{2}\frac{\lambda+\mu}{\lambda+3\mu} \left(\begin{array}{cc}1&-\bfi\\-\bfi&-1\end{array}\right)\bfa. 
\end{equation*}
Observe that 
\begin{equation*}
\begin{aligned}
\tilde{\bfa}+\tilde{\bfb} &= \left[ \id + \frac{1}{2}\frac{\lambda+\mu}{\lambda+3\mu} \left(\begin{array}{cc}1&-\bfi\\-\bfi&-1\end{array}\right) \right]\bfa + \left[ \id + \frac{1}{2}\frac{\lambda+\mu}{\lambda+3\mu} \left(\begin{array}{cc}1&\bfi\\\bfi&-1\end{array}\right) \right]\bfb, \\ 
\tilde{\bfa}-\tilde{\bfb} &= \left[ \id - \frac{1}{2}\frac{\lambda+\mu}{\lambda+3\mu} \left(\begin{array}{cc}1&-\bfi\\-\bfi&-1\end{array}\right) \right]\bfa - \left[ \id - \frac{1}{2}\frac{\lambda+\mu}{\lambda+3\mu} \left(\begin{array}{cc}1&\bfi\\\bfi&-1\end{array}\right) \right]\bfb. 
\end{aligned}
\end{equation*}
Hence, by \cite[Proposition~2.8.3]{Bernstein2009MatrixFacts} and the fact that $\frac{\lambda+\mu}{\lambda+3\mu}\neq\pm1$, the above system is uniquely solvable. Therefore, every harmonic polynomial of degree $m=1$ admits the representation \eqref{eq:H}. 
\end{remark}

\begin{remark}\label{rem:representation-H-higher-order}
Since 
\begin{equation*}
\begin{aligned}
& \left[ \id - \frac{1}{2}\left(\frac{\lambda+\mu}{\lambda+3\mu}\right)^{2}\left(\begin{array}{cc}1&-\bfi\\\bfi&1\end{array}\right) \right]\left[ \id - \frac{1}{2}\left(\frac{\lambda+\mu}{\lambda+3\mu}\right)^{2}\left(\begin{array}{cc}1&\bfi\\-\bfi&1\end{array}\right) \right] \\ 
&\quad = \overbrace{\left( 1 - \left(\frac{\lambda+\mu}{\lambda+3\mu}\right)^{2} \right)}^{\neq 0}\id
\end{aligned}
\end{equation*}
both 
\begin{equation*}
\left[ \id - \frac{1}{2}\left(\frac{\lambda+\mu}{\lambda+3\mu}\right)^{2}\left(\begin{array}{cc}1&-\bfi\\\bfi&1\end{array}\right) \right] \quad\text{and}\quad \left[ \id - \frac{1}{2}\left(\frac{\lambda+\mu}{\lambda+3\mu}\right)^{2}\left(\begin{array}{cc}1&\bfi\\-\bfi&1\end{array}\right) \right]
\end{equation*}
are invertible. Hence, every harmonic polynomial of degree $m\ge 2$ can be represented in the form \eqref{eq:H-higher-order}. 
\end{remark}

\begin{remark*}
In the degenerate case $\lambda+\mu=0$, equation \eqref{eq:general-solution} reduces to $\Delta w = \frac{1}{\mu}H$, so that the components of $w$ decouple and no longer interact. In this case, the general solution described in \Cref{lem:general-solution,lem:general-solution-higher-order} reduce to 
\begin{equation}
\begin{aligned}
w(re^{\bfi\theta}) &= \bfc r^{m+2}e^{-\bfi(m+2)\theta} + \bfd r^{m+2}e^{\bfi(m+2)\theta} \\ 
&\quad + \frac{\bfa}{4\mu(m+1)} r^{m+2}e^{-\bfi m\theta} + \frac{\bfb}{4\mu(m+1)} r^{m+2}e^{\bfi m\theta}. 
\end{aligned} \label{eq:degenerate-general-solution}
\end{equation} 
\end{remark*}

\begin{proof}[Proof of \Cref{lem:general-solution,lem:general-solution-higher-order}] 
We write $z=re^{\bfi\theta}$ and introduce the Wirtinger derivatives 
\begin{equation*}
\partial_{z}=\frac{1}{2}\left(\partial_{x}-\bfi\partial_{y}\right) ,\quad \partial_{\overline{z}}=\frac{1}{2}\left(\partial_{x}+\bfi\partial_{y}\right). 
\end{equation*}
Let $\mL = \mL^{\calC_{0}}$ and $\partial_{z}\partial_{\overline{z}}w = \begin{pmatrix} \partial_{z}\partial_{\overline{z}}w_1 \\ \partial_{z}\partial_{\overline{z}}w_2 \end{pmatrix}$ etc. A direct computation shows that 
\begin{equation*}
\begin{aligned}
\mL w &= \mu \Delta w + (\lambda+\mu)\nabla(\nabla\cdot w) \\ 
 \quad &= 2(\lambda+3\mu) \partial_{z}\partial_{\overline{z}}w + (\lambda+\mu)\left(\begin{array}{cc}1&\bfi\\\bfi&-1\end{array}\right)\partial_{z}^{2}w + (\lambda+\mu)\left(\begin{array}{cc}1&-\bfi\\-\bfi&-1\end{array}\right)\partial_{\overline{z}}^{2}w. 
\end{aligned}
\end{equation*}
Since 
\begin{equation*}
\begin{aligned}
&\mL(\overline{z}^{m+2}\bfc) = (\lambda+\mu)(m+2)(m+1)\overline{z}^{m}\left(\begin{array}{cc}1&-\bfi\\-\bfi&-1\end{array}\right)\bfc, \\ 
&\mL\left(\overline{z}^{m+1}z\left(\begin{array}{cc}1&-\bfi\\-\bfi&-1\end{array}\right)\bfc\right) = 2(\lambda+3\mu)(m+1)\overline{z}^{m}\left(\begin{array}{cc}1&-\bfi\\-\bfi&-1\end{array}\right)\bfc
\end{aligned}
\end{equation*}
it follows that 
\begin{equation*}
\mL\left( \overline{z}^{m+2}\bfc - \frac{\lambda+\mu}{2(\lambda+3\mu)}(m+2)\overline{z}^{m+1}z\left(\begin{array}{cc}1&-\bfi\\-\bfi&-1\end{array}\right)\bfc \right) = 0. 
\end{equation*}
Analogous arguments yield
\begin{equation*}
\mL\left( z^{m+2}\bfd - \frac{\lambda+\mu}{2(\lambda+3\mu)}(m+2)z^{m+1}\overline{z}\left(\begin{array}{cc}1&\bfi\\\bfi&-1\end{array}\right)\bfd \right) = 0. 
\end{equation*}
Meanwhile, we compute that 
\begin{equation}
\begin{aligned}
& \mL(\overline{z}^{m+1}z\bfa) = 2(\lambda+3\mu)(m+1)\overline{z}^{m}\bfa + (\lambda+\mu)(m+1)m\overline{z}^{m-1}z\left(\begin{array}{cc}1&-\bfi\\-\bfi&-1\end{array}\right)\bfa, \\ 
& \mL(\overline{z}z^{m+1}\bfb) = 2(\lambda+3\mu)(m+1)z^{m}\bfb + (\lambda+\mu)(m+1)m\overline{z}z^{m-1}\left(\begin{array}{cc}1&\bfi\\\bfi&-1\end{array}\right)\bfb,
\end{aligned} \label{eq:nonhomogeneous1} 
\end{equation}
then 
\begin{equation*}
\begin{aligned}
& \frac{1}{2(\lambda+3\mu)(m+1)} \mL(\overline{z}^{m+1}z\bfa + \overline{z}z^{m+1}\bfb) \\ 
&\quad = \overline{z}^{m}\left[ \bfa + \frac{(\lambda+\mu)m}{2(\lambda+3\mu)}\overline{z}^{-m+1}z^{m-1}\left(\begin{array}{cc}1&\bfi\\\bfi&-1\end{array}\right)\bfb \right] \\ 
&\qquad + z^{m}\left[ \bfb + \frac{(\lambda+\mu)m}{2(\lambda+3\mu)}\overline{z}^{m-1}z^{-m+1}\left(\begin{array}{cc}1&-\bfi\\-\bfi&-1\end{array}\right)\bfa \right]. 
\end{aligned}
\end{equation*}
By considering $m\in\{0,1\}$, we conclude the proof of \Cref{lem:general-solution}. 
We next compute that 
\begin{equation*}
\begin{aligned}
\mL\left( \overline{z}^{m}z^{2}\left(\begin{array}{cc}1&-\bfi\\-\bfi&-1\end{array}\right)\bfa \right) &= 4(\lambda+3\mu)m\overline{z}^{m-1}z\left(\begin{array}{cc}1&-\bfi\\-\bfi&-1\end{array}\right)\bfa \\
&\qquad + 4(\lambda+\mu)\overline{z}^{m}\left(\begin{array}{cc}1&-\bfi\\\bfi&1\end{array}\right)\bfa, 
\end{aligned}
\end{equation*}
and 
\begin{equation*}
\begin{aligned} 
\mL\left( \overline{z}^{2}z^{m}\left(\begin{array}{cc}1&\bfi\\\bfi&-1\end{array}\right)\bfb \right) &= 4(\lambda+3\mu)m\overline{z}z^{m-1}\left(\begin{array}{cc}1&\bfi\\\bfi&-1\end{array}\right)\bfb \\
&\qquad+ 4(\lambda+\mu)z^{m}\left(\begin{array}{cc}1&\bfi\\-\bfi&1\end{array}\right)\bfb. 
\end{aligned}
\end{equation*}
Combining this with \eqref{eq:nonhomogeneous1}, we conclude the proof of \Cref{lem:general-solution-higher-order}. 
\end{proof}

We now give a characterization of $v$ in the following lemmas. 

\begin{lemma}\label{lem:explicit-solution-bvp}
Let $m\in\{0,1\}$. Assume that $H$ admits the $\theta$-Fourier representation \eqref{eq:H} 
for some constant vectors $\bfa,\bfb\in\mC^{2}$. Let $0 \le \theta_{1} < \theta_{2} < 2\pi$, $\theta_{0}\in[\theta_{1},\theta_{2}]$ and write 
\begin{equation*}
\mathfrak{C} := \left\{ re^{\bfi\theta} : r>0 , \, \theta \in (\theta_{1},\theta_{2})\setminus \{ \theta_{0} \} \right\}. 
\end{equation*}
Write $z_{0}=e^{\bfi\theta_{0}}$. Suppose that $v$ solves 
\begin{equation*}
\mu\Delta v + (\lambda+\mu)\nabla(\nabla\cdot v) = H \text{ in $\mathfrak{C}$} ,\quad v|_{\theta=\theta_{0}}=\partial_{\theta}v|_{\theta=\theta_{0}}=0. 
\end{equation*}
Then the solution $v$ can be characterized as in \Cref{lem:general-solution} with 
\begin{equation}
\begin{aligned}
& \overbrace{\scalebox{0.8}{$\left(\begin{array}{cc} 
z_{0}^{-(m+2)}\id - \frac{\lambda+\mu}{2(\lambda+3\mu)}(m+2)z_{0}^{-m}\left(\begin{array}{cc}1&-\bfi\\-\bfi&-1\end{array}\right) & z_{0}^{m+2}\id - \frac{\lambda+\mu}{2(\lambda+3\mu)}(m+2)z_{0}^{m}\left(\begin{array}{cc}1&\bfi\\\bfi&-1\end{array}\right) \\ 
z_{0}^{-(m+2)}\id - \frac{\lambda+\mu}{2(\lambda+3\mu)}mz_{0}^{-m}\left(\begin{array}{cc}1&-\bfi\\-\bfi&-1\end{array}\right) & -\left[z_{0}^{m+2}\id - \frac{\lambda+\mu}{2(\lambda+3\mu)}mz_{0}^{m}\left(\begin{array}{cc}1&\bfi\\\bfi&-1\end{array}\right)\right] 
\end{array}\right)$}}^{\text{with determinant $4-4\left(\frac{\lambda+\mu}{\lambda+3\mu}\right)^{2}\neq 0$}}
\left(\begin{array}{c}\bfc\\\bfd\end{array}\right) \\ 
&\quad =-  \left(\begin{array}{cc} 
\frac{1}{2(\lambda+3\mu)(m+1)}z_{0}^{-m}\id & \frac{1}{2(\lambda+3\mu)(m+1)}z_{0}^{m}\id \\ 
\frac{m}{2(\lambda+3\mu)(m+1)(m+2)}z_{0}^{-m}\id & -\frac{m}{2(\lambda+3\mu)(m+1)(m+2)}z_{0}^{m}\id 
\end{array}\right) \left(\begin{array}{c}\bfa\\\bfb\end{array}\right)
\end{aligned} \label{eq:matrix-equation1} 
\end{equation} 
\end{lemma} 

\begin{lemma}\label{lem:explicit-solution-bvp-higher-order} 
Let $m\ge 2$ be an integer. Assume that $H$ admits the $\theta$-Fourier representation \eqref{eq:H-higher-order} for some constant vectors $\bfa,\bfb\in\mC^{2}$. Let $0 \le \theta_{1} < \theta_{2} < 2\pi$, $\theta_{0}\in[\theta_{1},\theta_{2}]$ and write 
\begin{equation*}
\mathfrak{C} := \left\{ re^{\bfi\theta} : r>0 , \theta \in (\theta_{1},\theta_{2})\setminus\{\theta_{0}\} \right\}. 
\end{equation*}
Write $z_{0}=e^{\bfi\theta_{0}}$. Suppose that $v$ solves 
\begin{equation*}
\mu\Delta v + (\lambda+\mu)\nabla(\nabla\cdot v) = H \text{ in $\mathfrak{C}$} ,\quad v|_{\theta=\theta_{0}}=\partial_{\theta}v|_{\theta=\theta_{0}}=0. 
\end{equation*}
Then the solution $v$ can be characterized as in \Cref{lem:general-solution-higher-order} with 
\begin{equation}
\begin{aligned}
& \overbrace{\scalebox{0.8}{$\left(\begin{array}{cc} 
z_{0}^{-(m+2)}\id - \frac{\lambda+\mu}{2(\lambda+3\mu)}(m+2)z_{0}^{-m}\left(\begin{array}{cc}1&-\bfi\\-\bfi&-1\end{array}\right) & z_{0}^{m+2}\id - \frac{\lambda+\mu}{2(\lambda+3\mu)}(m+2)z_{0}^{m}\left(\begin{array}{cc}1&\bfi\\\bfi&-1\end{array}\right) \\ 
z_{0}^{-(m+2)}\id - \frac{\lambda+\mu}{2(\lambda+3\mu)}mz_{0}^{-m}\left(\begin{array}{cc}1&-\bfi\\-\bfi&-1\end{array}\right) & -\left[z_{0}^{m+2}\id - \frac{\lambda+\mu}{2(\lambda+3\mu)}mz_{0}^{m}\left(\begin{array}{cc}1&\bfi\\\bfi&-1\end{array}\right)\right] 
\end{array}\right)$}}^{\text{with determinant $4-4\left(\frac{\lambda+\mu}{\lambda+3\mu}\right)^{2}\neq 0$}}
\left(\begin{array}{c}\bfc\\\bfd\end{array}\right) \\ 
& = \scalebox{0.65}{$\left(\begin{array}{cc} 
z_{0}^{-m}\id - \frac{\lambda+\mu}{4(\lambda+3\mu)}(m+1)z_{0}^{-(m-2)}\left(\begin{array}{cc}1&-\bfi\\-\bfi&-1\end{array}\right) & z_{0}^{m}\id - \frac{\lambda+\mu}{4(\lambda+3\mu)}(m+1)z_{0}^{m-2}\left(\begin{array}{cc}1&\bfi\\\bfi&-1\end{array}\right) \\ 
\frac{m}{m+2}z_{0}^{-m}\id - \frac{\lambda+\mu}{4(\lambda+3\mu)}\frac{(m+1)(m-2)}{m+2}z_{0}^{-(m-2)}\left(\begin{array}{cc}1&-\bfi\\-\bfi&-1\end{array}\right) & -\left[\frac{m}{m+2}z_{0}^{m}\id - \frac{\lambda+\mu}{4(\lambda+3\mu)}\frac{(m+1)(m-2)}{m+2}z_{0}^{m-2}\left(\begin{array}{cc}1&\bfi\\\bfi&-1\end{array}\right)\right] 
\end{array}\right)$} \left(\begin{array}{c}\hat{\bfa}\\ \hat{\bfb}\end{array}\right). 
\end{aligned} \label{eq:matrix-equation1-higher-order} 
\end{equation} 
where $\hat{\bfa}=-\frac{1}{2(\lambda+3\mu)(m+1)}\bfa$ and $\hat{\bfb}=-\frac{1}{2(\lambda+3\mu)(m+1)}\bfb$
\end{lemma}

\begin{remark*}
By the unique continuation principle (UCP) for the isotropic elasticity system \cite{DLW20UCPLame,LNUW11UCPLame,LW05UCPLame}, the solution $v$ to \eqref{eq:Lame-2D} must take the form given in \Cref{lem:explicit-solution-bvp,lem:explicit-solution-bvp-higher-order} with $\theta_{0}=0$. 
\end{remark*}

\begin{remark*}
In the degenerate case $\lambda+\mu=0$, the equation \eqref{eq:matrix-equation1} with $\theta_{0}=0$ reduces to 
\begin{equation*}
\bfc = -\frac{1}{4\mu(m+1)}\left(\frac{m+1}{m+2}\bfa+\frac{1}{m+2}\bfb\right) ,\quad \bfd = -\frac{1}{4\mu(m+1)}\left(\frac{1}{m+2}\bfa+\frac{m+1}{m+2}\bfb\right).  
\end{equation*}
Substituting these coefficients into \eqref{eq:degenerate-general-solution} yields 
\begin{equation*}
\begin{aligned}
v(re^{\bfi\theta}) &= \frac{1}{4\mu(m+1)}\left[ r^{m+2}e^{-\bfi m\theta} - \frac{1}{m+2}r^{m+2}e^{\bfi(m+2)\theta} -\frac{m+1}{m+2}r^{m+2}e^{-\bfi(m+2)\theta} \right]\bfa \\ 
&\quad + \frac{1}{4\mu(m+1)}\left[ r^{m+2}e^{\bfi m\theta} - \frac{m+1}{m+2}r^{m+2}e^{\bfi(m+2)\theta} -\frac{1}{m+2}r^{m+2}e^{-\bfi(m+2)\theta} \right]\bfb 
\end{aligned} 
\end{equation*} 
which coincides with the expression obtained in \cite[Lemma~3.5]{SS25vanishingcontrast}. 
\end{remark*}

\begin{proof}[Proof of \Cref{lem:explicit-solution-bvp,lem:explicit-solution-bvp-higher-order}] 
Inserting the general solution in \Cref{lem:general-solution,lem:general-solution-higher-order} into the boundary conditions $v(re^{\bfi\theta})|_{\theta=\theta_{0}}=\partial_{\theta}v(re^{\bfi\theta})|_{\theta=\theta_{0}} = 0$, we obtain 
\begin{subequations}\label{eq:matrix-equation1-equation-form} 
\begin{equation}
\begin{aligned}
0 &= \left[ z_{0}^{-(m+2)}\id - \frac{\lambda+\mu}{2(\lambda+3\mu)}(m+2)z_{0}^{-m}\left(\begin{array}{cc}1&-\bfi\\-\bfi&-1\end{array}\right)\right]\bfc \\ 
&\quad + \left[ z_{0}^{m+2}\id - \frac{\lambda+\mu}{2(\lambda+3\mu)}(m+2)z_{0}^{m}\left(\begin{array}{cc}1&\bfi\\\bfi&-1\end{array}\right)\right]\bfd \\ 
&\quad + \frac{1}{2(\lambda+3\mu)(m+1)}(z_{0}^{-m}\bfa+z_{0}^{m}\bfb)
\end{aligned}
\end{equation}
and 
\begin{equation}
\begin{aligned}
0 &= \left[ -(m+2)z_{0}^{-(m+2)}\id + m\frac{\lambda+\mu}{2(\lambda+3\mu)}(m+2)z_{0}^{-m}\left(\begin{array}{cc}1&-\bfi\\-\bfi&-1\end{array}\right)\right]\bfc \\ 
&\quad + \left[ (m+2)z_{0}^{m+2}\id - m\frac{\lambda+\mu}{2(\lambda+3\mu)}(m+2)z_{0}^{m}\left(\begin{array}{cc}1&\bfi\\\bfi&-1\end{array}\right)\right]\bfd \\ 
&\quad + \frac{1}{2(\lambda+3\mu)(m+1)}(-mz_{0}^{-m}\bfa+mz_{0}^{m}\bfb)
\end{aligned}
\end{equation}
\end{subequations}
for $m\in\{0,1\}$. For $m\ge 2$, we obtain 
\begin{subequations}\label{eq:matrix-equation1-higher-order-equation-form} 
\begin{equation}
\begin{aligned}
0 &= \left[ z_{0}^{-(m+2)}\id - \frac{\lambda+\mu}{2(\lambda+3\mu)}(m+2)z_{0}^{-m}\left(\begin{array}{cc}1&-\bfi\\-\bfi&-1\end{array}\right)\right]\bfc \\ 
&\quad + \left[ z_{0}^{m+2}\id - \frac{\lambda+\mu}{2(\lambda+3\mu)}(m+2)z_{0}^{m}\left(\begin{array}{cc}1&\bfi\\\bfi&-1\end{array}\right)\right]\bfd \\ 
&\quad + \frac{1}{2(\lambda+3\mu)(m+1)}\left[ z_{0}^{-m}\id - \frac{\lambda+\mu}{4(\lambda+3\mu)}(m+1)z_{0}^{-(m-2)}\left(\begin{array}{cc}1&-\bfi\\-\bfi&-1\end{array}\right) \right]\bfa \\ 
&\quad + \frac{1}{2(\lambda+3\mu)(m+1)}\left[ z_{0}^{m}\id - \frac{\lambda+\mu}{4(\lambda+3\mu)}(m+1)z_{0}^{m-2}\left(\begin{array}{cc}1&\bfi\\\bfi&-1\end{array}\right) \right]\bfb  
\end{aligned}
\end{equation}
and 
\begin{equation}
\scalebox{0.9}{$\begin{aligned}
0 &= \left[ -(m+2)z_{0}^{-(m+2)}\id + m\frac{\lambda+\mu}{2(\lambda+3\mu)}(m+2)z_{0}^{-m}\left(\begin{array}{cc}1&-\bfi\\-\bfi&-1\end{array}\right)\right]\bfc \\ 
&\quad + \left[ (m+2)z_{0}^{m+2}\id - m\frac{\lambda+\mu}{2(\lambda+3\mu)}(m+2)z_{0}^{m}\left(\begin{array}{cc}1&\bfi\\\bfi&-1\end{array}\right)\right]\bfd \\ 
&\quad + \frac{1}{2(\lambda+3\mu)(m+1)}\left[ -mz_{0}^{-m}\id + (m-2)\frac{\lambda+\mu}{4(\lambda+3\mu)}(m+1)z_{0}^{-(m-2)}\left(\begin{array}{cc}1&-\bfi\\-\bfi&-1\end{array}\right) \right]\bfa \\ 
&\quad + \frac{1}{2(\lambda+3\mu)(m+1)}\left[ mz_{0}^{m}\id - (m-2)\frac{\lambda+\mu}{4(\lambda+3\mu)}(m+1)z_{0}^{m-2}\left(\begin{array}{cc}1&\bfi\\\bfi&-1\end{array}\right) \right]\bfb.  
\end{aligned}$} 
\end{equation}
\end{subequations}
Note that the block matrix equations \eqref{eq:matrix-equation1} and \eqref{eq:matrix-equation1-higher-order} are equivalent to \eqref{eq:matrix-equation1-equation-form} and \eqref{eq:matrix-equation1-higher-order-equation-form}, respectively. 

Finally, by \cite[Proposition~2.8.4]{Bernstein2009MatrixFacts}, the determinants of the matrices on the left-hand sides of \eqref{eq:matrix-equation1-equation-form} and \eqref{eq:matrix-equation1-higher-order-equation-form} can be computed as follows: 
\allowdisplaybreaks 
\begin{align*}
& \det \scalebox{0.9}{$\left(\begin{array}{cc} 
z_{0}^{-(m+2)}\id - \frac{\lambda+\mu}{2(\lambda+3\mu)}(m+2)z_{0}^{-m}\left(\begin{array}{cc}1&-\bfi\\-\bfi&-1\end{array}\right) & z_{0}^{m+2}\id - \frac{\lambda+\mu}{2(\lambda+3\mu)}(m+2)z_{0}^{m}\left(\begin{array}{cc}1&\bfi\\\bfi&-1\end{array}\right) \\ 
z_{0}^{-(m+2)}\id - \frac{\lambda+\mu}{2(\lambda+3\mu)}mz_{0}^{-m}\left(\begin{array}{cc}1&-\bfi\\-\bfi&-1\end{array}\right) & -\left[z_{0}^{m+2}\id - \frac{\lambda+\mu}{2(\lambda+3\mu)}mz_{0}^{m}\left(\begin{array}{cc}1&\bfi\\\bfi&-1\end{array}\right)\right] 
\end{array}\right)$}  \\ 
&\quad = \det \left[ z_{0}^{m+2}\id - m\frac{\lambda+\mu}{2(\lambda+3\mu)}z_{0}^{m}\left(\begin{array}{cc}1&\bfi\\\bfi&-1\end{array}\right) \right] \times \\ 
&\qquad \times \det \scalebox{0.75}{$\left[ \begin{aligned}
& z_{0}^{-(m+2)}\id - \frac{\lambda+\mu}{2(\lambda+3\mu)}(m+2)z_{0}^{-m}\left(\begin{array}{cc}1&-\bfi\\-\bfi&-1\end{array}\right) \\ 
& +\left[ z_{0}^{m+2}\id - \frac{\lambda+\mu}{2(\lambda+3\mu)}(m+2)z_{0}^{m}\left(\begin{array}{cc}1&\bfi\\\bfi&-1\end{array}\right) \right]\left[ z_{0}^{m+2}\id - \frac{\lambda+\mu}{2(\lambda+3\mu)}mz_{0}^{m}\left(\begin{array}{cc}1&\bfi\\\bfi&-1\end{array}\right) \right]^{-1}\times\\
&\qquad \times \left[ z_{0}^{-(m+2)}\id - \frac{\lambda+\mu}{2(\lambda+3\mu)}mz_{0}^{-m}\left(\begin{array}{cc}1&-\bfi\\-\bfi&-1\end{array}\right) \right] \end{aligned} \right]$} \\ 
&\quad = \overbrace{\det \scalebox{0.8}{$\left[ z_{0}^{m+2}\id - m\frac{\lambda+\mu}{2(\lambda+3\mu)}z_{0}^{m}\left(\begin{array}{cc}1&\bfi\\\bfi&-1\end{array}\right) \right]$}}^{=z_{0}^{2(m+2)}} \overbrace{\det \scalebox{0.8}{$\left[ z_{0}^{-(m+2)}\id - \frac{\lambda+\mu}{2(\lambda+3\mu)}mz_{0}^{-m}\left(\begin{array}{cc}1&-\bfi\\-\bfi&-1\end{array}\right) \right]$}}^{=z_{0}^{-2(m+2)}} \times \\ 
&\qquad \times \det \scalebox{0.75}{$\left[ \begin{aligned}
& \left[z_{0}^{-(m+2)}\id - \frac{\lambda+\mu}{2(\lambda+3\mu)}(m+2)z_{0}^{-m}\left(\begin{array}{cc}1&-\bfi\\-\bfi&-1\end{array}\right)\right]\left[ z_{0}^{-(m+2)}\id - \frac{\lambda+\mu}{2(\lambda+3\mu)}mz_{0}^{-m}\left(\begin{array}{cc}1&-\bfi\\-\bfi&-1\end{array}\right) \right]^{-1} \\ 
& +\left[ z_{0}^{m+2}\id - \frac{\lambda+\mu}{2(\lambda+3\mu)}(m+2)z_{0}^{m}\left(\begin{array}{cc}1&\bfi\\\bfi&-1\end{array}\right) \right]\left[ z_{0}^{m+2}\id - \frac{\lambda+\mu}{2(\lambda+3\mu)}mz_{0}^{m}\left(\begin{array}{cc}1&\bfi\\\bfi&-1\end{array}\right) \right]^{-1}  \end{aligned} \right]$} \\ 
&\quad = \det \scalebox{0.75}{$\left[ \begin{aligned}
& \left[z_{0}^{-(m+2)}\id - \frac{\lambda+\mu}{2(\lambda+3\mu)}(m+2)z_{0}^{-m}\left(\begin{array}{cc}1&-\bfi\\-\bfi&-1\end{array}\right)\right]\left[ z_{0}^{m+2}\id + \frac{\lambda+\mu}{2(\lambda+3\mu)}mz_{0}^{m+4}\left(\begin{array}{cc}1&-\bfi\\-\bfi&-1\end{array}\right) \right] \\ 
& +\left[ z_{0}^{m+2}\id - \frac{\lambda+\mu}{2(\lambda+3\mu)}(m+2)z_{0}^{m}\left(\begin{array}{cc}1&\bfi\\\bfi&-1\end{array}\right) \right]\left[ z_{0}^{-(m+2)}\id + \frac{\lambda+\mu}{2(\lambda+3\mu)}mz_{0}^{-m-4}\left(\begin{array}{cc}1&\bfi\\\bfi&-1\end{array}\right) \right] \end{aligned} \right]$} \\ 
&\quad = \det\left[2\id - \frac{\lambda+\mu}{\lambda+3\mu}z_{0}^{2}\left(\begin{array}{cc}1&-\bfi\\-\bfi&-1\end{array}\right) - \frac{\lambda+\mu}{\lambda+3\mu}z_{0}^{-2}\left(\begin{array}{cc}1&\bfi\\\bfi&-1\end{array}\right) \right] \\ 
&\quad = 4 - 4\left(\frac{\lambda+\mu}{\lambda+3\mu}\right)^{2} \neq 0, 
\end{align*}
which concludes our lemma. 
\end{proof}

We now show that there are no blowup solutions supported in sectors with opening angle different from $\pi$. This result is essential for the characterization of blowup solutions in two dimensions. 

\begin{lemma}\label{lem:half-space-solution-uniqueness}
Let $m\in\{0,1\}$. Assume that $H$ admits the $\theta$-Fourier representation \eqref{eq:H} 
for some constant vectors $\bfa,\bfb\in\mC^{2}$. 
Let $\theta_0\in(0,2\pi)\setminus\{\pi\}$, and write $\mathfrak{C}=\left\{re^{\bfi\theta}:r>0,\theta\in(0,\theta_0)\right\}$. Suppose that $v$ solves 
\begin{equation}
\mu\Delta v + (\lambda+\mu)\nabla(\nabla\cdot v) = H \text{ in $\mathfrak{C}$} ,\quad  v|_{\theta=\theta_{0}}=\partial_{\theta}v|_{\theta=\theta_{0}} = 0. \label{eq:Lame-2D-sector-new} 
\end{equation} 
If 
\begin{equation}
v|_{\theta=0} = \partial_{\theta}v|_{\theta=0} = 0,  \label{eq:Dirichlet-sector1-new}
\end{equation}
then $v\equiv 0$ and $H\equiv 0$. 
\end{lemma}

\begin{remark}\label{rem:difficulty-higher-order}
We believe that \Cref{lem:half-space-solution-uniqueness} remains valid for $m\ge 2$, with \Cref{lem:explicit-solution-bvp-higher-order} replacing \Cref{lem:explicit-solution-bvp}. In principle, this could be established through direct computation. However, the calculations become lengthy, so we do not pursue this direction further. 
\end{remark}

\begin{proof}[Proof of \Cref{lem:half-space-solution-uniqueness}] 
As mentioned above, by the unique continuation principle (UCP) for the isotropic elasticity system \cite{DLW20UCPLame,LNUW11UCPLame,LW05UCPLame}, such a solution $v$ must take the form given in \Cref{lem:explicit-solution-bvp}. Throughout the proof, we write $z_{0}=e^{\bfi\theta_{0}}$. 
Note that \eqref{eq:Dirichlet-sector1-new} equivalent to 
\begin{subequations} \label{eq:Dirichlet-sector1-new-equivalent} 
\begin{equation}
\begin{aligned}
& \left[ \id -\frac{\lambda+\mu}{2(\lambda +3\mu)}(m+2)\left(\begin{array}{cc}1&-\bfi \\ -\bfi&-1\end{array}\right) \right]\bfc + \left[\id - \frac{\lambda+\mu}{2(\lambda +3\mu)}(m+2)\left(\begin{array}{cc}1&\bfi \\ \bfi&-1\end{array}\right)\right]\bfd \\ 
&\quad =- \frac{1}{2(\lambda+3\mu)(m+1)}(\bfa+\bfb). 
\end{aligned} 
\end{equation} 
and 
\begin{equation}
\begin{aligned}
& \left[ -(m+2)\id + m\frac{\lambda+\mu}{2(\lambda +3\mu)}(m+2)\left(\begin{array}{cc}1&-\bfi \\ -\bfi&-1\end{array}\right) \right]\bfc \\
&\qquad + \left[(m+2)\id -  m\frac{\lambda+\mu}{2(\lambda +3\mu)}(m+2)\left(\begin{array}{cc}1&\bfi \\ \bfi&-1\end{array}\right)\right]\bfd \\ 
&\quad =- \frac{1}{2(\lambda+3\mu)(m+1)}(-m\bfa+ m\bfb). 
\end{aligned} 
\end{equation} 
\end{subequations} 

\medskip 

\noindent \emph{We begin with the case $m=0$.} In this case, \eqref{eq:matrix-equation1} reduces to 
\begin{equation}
\begin{aligned} 
&\left(\begin{array}{cc} 
z_{0}^{-2}\id - \frac{\lambda+\mu}{\lambda+3\mu}\left(\begin{array}{cc}1&-\bfi\\-\bfi&-1\end{array}\right) & z_{0}^{2}\id - \frac{\lambda+\mu}{\lambda+3\mu}\left(\begin{array}{cc}1&\bfi\\\bfi&-1\end{array}\right) \\ 
z_{0}^{-2}\id & -z_{0}^{2}\id 
\end{array}\right)
\left(\begin{array}{c}\bfc\\\bfd\end{array}\right) \\
&\quad = - \frac{1}{2(\lambda+3\mu)} \left(\begin{array}{c} \bfa + \bfb \\ 0 \end{array}\right). 
\end{aligned} \label{eq:matrix-equation1-m0}
\end{equation}
Since \eqref{eq:Dirichlet-sector1-new-equivalent} reads 
\begin{equation*}
\left[ \id -\frac{\lambda+\mu}{\lambda +3\mu}\left(\begin{array}{cc}1&-\bfi \\ -\bfi&-1\end{array}\right) \right]\bfc + \left[\id - \frac{\lambda+\mu}{\lambda +3\mu}\left(\begin{array}{cc}1&\bfi \\ \bfi&-1\end{array}\right)\right]\bfd  =- \frac{1}{2(\lambda+3\mu)}(\bfa+\bfb) 
\end{equation*}
and 
\begin{equation*}
\bfc - \bfd = 0, 
\end{equation*} 
then, by \eqref{eq:matrix-equation1-m0}, we have 
\begin{equation*}
\overbrace{\left(\begin{array}{cc} 
(z_{0}^{-2}-1)\id & (z_{0}^{2}-1)\id  \\ 
(z_{0}^{-2}-1)\id & -(z_{0}^{2}-1)\id 
\end{array}\right)}^{\text{with determinant $4(z_{0}-z_{0}^{-1})^{4}\neq 0$}}
\left(\begin{array}{c}\bfc\\\bfd\end{array}\right) = 0, 
\end{equation*} 
which yields $\bfc=\bfd=0$, and hence $v\equiv 0$ and $H\equiv 0$. 

\medskip 

\noindent \emph{We proceed to the case $m=1$.} In this case, \eqref{eq:matrix-equation1} reads 
\begin{equation*}
\begin{aligned}
& \left(\begin{array}{cc} 
z_{0}^{-3}\id - \frac{\lambda+\mu}{2(\lambda+3\mu)}3z_{0}^{-1}\left(\begin{array}{cc}1&-\bfi\\-\bfi&-1\end{array}\right) & z_{0}^{3}\id - \frac{\lambda+\mu}{2(\lambda+3\mu)}3z_{0}\left(\begin{array}{cc}1&\bfi\\\bfi&-1\end{array}\right) \\ 
z_{0}^{-3}\id - \frac{\lambda+\mu}{2(\lambda+3\mu)}z_{0}^{-1}\left(\begin{array}{cc}1&-\bfi\\-\bfi&-1\end{array}\right) & -\left[z_{0}^{3}\id - \frac{\lambda+\mu}{2(\lambda+3\mu)}z_{0}\left(\begin{array}{cc}1&\bfi\\\bfi&-1\end{array}\right)\right] 
\end{array}\right) 
\left(\begin{array}{c}\bfc\\\bfd\end{array}\right) \\ 
&\quad =- \frac{1}{12(\lambda+3\mu)}\left(\begin{array}{cc} 
3z_{0}^{-1}\id & 3z_{0}\id \\ 
z_{0}^{-1}\id & -z_{0}\id 
\end{array}\right) \left(\begin{array}{c}\bfa\\\bfb\end{array}\right). 
\end{aligned} 
\end{equation*} 
Multiplying the above equation with $\left(\begin{array}{cc} 
z_{0}\id & 3z_{0}\id \\ 
z_{0}^{-1}\id & -3z_{0}^{-1}\id 
\end{array}\right)$, we then reach 
\begin{equation}
\begin{aligned}
& \left(\begin{array}{cc} 
4z_{0}^{-2}\id - \frac{\lambda+\mu}{2(\lambda+3\mu)}6\left(\begin{array}{cc}1&-\bfi\\-\bfi&-1\end{array}\right) & -2z_{0}^{4}\id \\ 
-2z_{0}^{-4}\id & 4z_{0}^{2}\id - \frac{\lambda+\mu}{2(\lambda+3\mu)}6\left(\begin{array}{cc}1&\bfi\\\bfi&-1\end{array}\right) 
\end{array}\right)
\left(\begin{array}{c}\bfc\\\bfd\end{array}\right) \\ 
&\quad = -\frac{1}{2(\lambda+3\mu)} \left(\begin{array}{c}\bfa\\\bfb\end{array}\right).
\end{aligned} \label{eq:matrix-equation1-m1} 
\end{equation} 
Since \eqref{eq:Dirichlet-sector1-new-equivalent} reads  
\begin{equation*}
\begin{aligned}
& \left[ \id -\frac{\lambda+\mu}{2(\lambda +3\mu)}3\left(\begin{array}{cc}1&-\bfi \\ -\bfi&-1\end{array}\right) \right]\bfc + \left[\id - \frac{\lambda+\mu}{2(\lambda +3\mu)}3\left(\begin{array}{cc}1&\bfi \\ \bfi&-1\end{array}\right)\right]\bfd \\ 
&\quad =- \frac{1}{4(\lambda+3\mu)}(\bfa+\bfb) 
\end{aligned} 
\end{equation*} 
and 
\begin{equation*}
\begin{aligned}
& \left[ -3\id + \frac{\lambda+\mu}{2(\lambda +3\mu)}3\left(\begin{array}{cc}1&-\bfi \\ -\bfi&-1\end{array}\right) \right]\bfc + \left[3\id -  \frac{\lambda+\mu}{2(\lambda +3\mu)}3\left(\begin{array}{cc}1&\bfi \\ \bfi&-1\end{array}\right)\right]\bfd \\ 
&\quad =- \frac{1}{4(\lambda+3\mu)}(-\bfa+ \bfb), 
\end{aligned} 
\end{equation*} 
we obtain  
\begin{equation*}
\left[ 4\id -\frac{\lambda+\mu}{2(\lambda +3\mu)}6\left(\begin{array}{cc}1&-\bfi \\ -\bfi&-1\end{array}\right) \right]\bfc -2 \bfd =- \frac{1}{2(\lambda+3\mu)}\bfa 
\end{equation*} 
and 
\begin{equation*}
-2\bfc + \left[4\id -  \frac{\lambda+\mu}{2(\lambda +3\mu)}6\left(\begin{array}{cc}1&\bfi \\ \bfi&-1\end{array}\right)\right]\bfd = - \frac{1}{2(\lambda+3\mu)}\bfb.
\end{equation*} 
Combining the above two identities with \eqref{eq:matrix-equation1-m1}, we obtain 
\begin{equation*}
\overbrace{\left(\begin{array}{cc} 
2(z_{0}^{-2}-1)\id & -(z_{0}^{4}-1)\id \\ 
-(z_{0}^{-4}-1)\id & 2(z_{0}^{2}-1)\id 
\end{array}\right)}^{\text{with determinant $16(z_{0}-z_{0}^{-1})^{8}\neq 0$}}
\left(\begin{array}{c}\bfc\\\bfd\end{array}\right) = 0 
\end{equation*}
which yields $\bfc=\bfd=0$, and hence $v\equiv 0$ and $H\equiv 0$. 
\end{proof}

We can now prove the free boundary regularity in two dimensions. 

\begin{proposition}\label{prop:fb-2d} 
Assume that $d=2$ and $m\in\{0,1\}$, that $\partial D$ is piecewise $C^{1}$, and $\calC_{0}$ is isotropic. If the assumptions of \Cref{lem:non-degeneracy} are satisfied, then the support of any blowup limit $v$ of order $m+2$ is a half space. 
\end{proposition}

\begin{proof}
Since $D$ has Lipschitz boundary, and $v$ vanishes in an exterior cone, we have $\supp\,(|v|)\neq\mR^{2}$. On the other hand, by \Cref{lem:non-degeneracy}, we see that $\supp\,(|v|)\neq\emptyset$. Thus after a rotation, any blowup limit $v$ satisfies the conditions in \Cref{lem:half-space-solution-uniqueness} for some $\theta_0 \in (0,2\pi)$. Now \Cref{lem:half-space-solution-uniqueness} implies that necessarily $\theta_0 = \pi$.
\end{proof}

With \Cref{lem:weak-flatness,prop:fb-2d} at hand, the following proposition can be proved by the same argument as in \cite[Proposition~2.8]{KSS25AnisotropicII}. 

\begin{proposition}\label{prop:fb-2d-result1} 
Assume that $d=2$ and $m\in\{0,1\}$, that $\partial D$ is piecewise $C^{1}$, and $\calC_{0}$ is isotropic. If the assumptions of \Cref{lem:non-degeneracy} are satisfied, then $\partial D$ is $C^{1}$ near $0$. 
\end{proposition}

We finally conclude \Cref{thm:1}.

\begin{proof}[Proof of \Cref{thm:1}]
Without loss of generality, we may assume that $x_{0}=0$. 
By \eqref{eq:non-degeneracy-condition}, the contrast $h$ is $C^{\alpha}$ near $0$ and satisfies $h(0)\neq 0$. 
Since $u^{\rm inc}\not\equiv0$ solves $(\mL^{\calC_{0}}+\kappa^{2})u^{\rm inc}=0$ in $\mR^{2}$, its analyticity \cite{MN57AnalyticSolutionEllipticSystems} implies that, for each $i=1,2$, there exists a nonnegative integer $m_{i}$ such that 
\begin{equation}
-u_{i}^{\rm inc} = H_{i} + \tilde{R}_{i}, \label{eq:ui-expansion}
\end{equation}
where $H_{i}$ is a homogeneous polynomial of order $m_{i}$ and $\abs{\tilde{R}_{i}(x)}\le C\abs{x}^{m_{i}+1}$. If the component $u_{i}^{\rm inc}\equiv0$, we simply set $m_{i}=+\infty$. 
In addition, both the divergence free and curl free parts of $u^{\rm inc}$ are such that each of their components solves a scalar Helmholtz equation (see \Cref{lem:decompose}). Thus, by a simple blowup argument, the first nontrivial homogeneous polynomial in the Taylor expansion of each of these components is harmonic. It follows from this that also each $H_i$ is harmonic.

Applying the argument of \cite[Lemma~2.5]{SS25vanishingcontrast} to each component, we obtain 
\begin{equation*}
-hu_{i}^{\rm inc} = h(0)H_{i} + R_{i} \quad\text{with}\quad \abs{R_{i}(x)}\le C\abs{x}^{m_{i}+\alpha}. 
\end{equation*}
Let $m=\min\{m_{1},m_{2}\}$. Using \eqref{eq:degenerate-condition}, we further deduce that 
\begin{equation*}
\abs{(\mL^{\tilde{\calC}-\calC_{0}}u^{\rm inc})(x)} \le C\abs{x}^{m+\alpha}, 
\end{equation*}
and hence 
\begin{equation*}
f_{i} = h(0)H_{i} + R_{i} \quad\text{with}\quad \abs{R_{i}(x)} \le C\abs{x}^{m+\alpha}. 
\end{equation*}
On the other hand, 
\begin{equation*}
\abs{g_{i}(x)} \le C\abs{x}^{m+1+\alpha}.
\end{equation*}
After this reduction, \Cref{thm:1} follows directly from the corresponding free boundary result in \Cref{prop:fb-2d-result1}. 
\end{proof}

\subsection{The higher-dimensional obstacle case} 

In this relatively simple geometric setting, the approach of \cite[Theorem~1.9]{KSS25AnisotropicII} applies and allows us to circumvent the introduction of a balanced energy functional and its monotonicity formula. We begin with the following lemma, whose proof follows the same line of argument as \cite[Lemma~3.4]{KSS25AnisotropicII}: 

\begin{lemma}\label{lem:high-dimen-poly}
Let $d\ge 3$, let 
\begin{equation*}
\tilde{\Gamma}_{\pm} = \left\{ re^{\bfi\theta_{\pm}}\in\mC\cong\mR^{2} : r\ge 0\right\}\times\mR^{d-2}
\end{equation*}
with distinct $\theta_{\pm}\in[0,2\pi)$, and let $\mathfrak{C}$ be a connected component of $\mR^d \setminus (\tilde{\Gamma}_{+}\cup\tilde{\Gamma}_{-})$.  Let $v$ be a solution to 
\begin{equation}
\mL^{\calC_{0}}v = P\chi_{\mathfrak{C}} \quad \text{in $B_{1}\subset\mR^{d}$} \label{eq:blowup-in-ball}
\end{equation}
and assume that $v=0$ in $B_{1}\setminus \ol{\mathfrak{C}}$, where $P\not\equiv0$, and that, for each $i$, the component $P_{i}$ is either identically zero or a homogeneous polynomial of degree $k$. Then $v$ is a homogeneous polynomial of order $k+2$ in $B_1 \cap \mathfrak{C}$. 
\end{lemma}

\begin{proof}
Since homogeneity of polynomials is preserved under orthogonal transformations, it follows from \Cref{rem:rotational-symmetry} that it suffices to prove the lemma in the case $\theta_{+}=\pi/2$, that is, 
\begin{equation*}
\tilde{\Gamma}_{+} = \{0\}\times\mR_{\ge 0} \times \mR^{d-2}. 
\end{equation*}

Let $\mathcal{P}_{k}$ denote the space of homogeneous polynomials on $\mR^d$ of degree $k$. As in \cite[Lemma~3.3]{SS25vanishingcontrast}, we observe that the operator 
\begin{equation*}
\mL^{\calC_{0}} : (x_{1}^{2}\mathcal{P}_{k})^{d} \rightarrow (\mathcal{P}_{k})^{d} 
\end{equation*}
is linear and injective (unique continuation holds for $\mL^{\calC_{0}}$ for instance by analyticity). Since it acts between finite-dimensional spaces of the same dimension, it is also surjective. It follows that there is $w \in (x_{1}^{2}\mathcal{P}_{k})^{d}$ such that $\mL^{\calC_{0}} w = P$ in $\mR^d$.

Now $u = v - w$ satisfies 
\[
\mL^{\calC_{0}} u = 0 \text{ in $B_1 \cap \mathfrak{C}$}, \qquad u|_{x_1=0} = \p_{x_1} u|_{x_1=0} = 0.
\]
Again by unique continuation $u = 0$ in $B_1 \cap \mathfrak{C}$, which proves the lemma.
\end{proof}

Having established \Cref{lem:high-dimen-poly}, we are now in a position to prove \Cref{thm:2} by adapting Federer's dimension reduction argument, following \cite[Theorem~1.9]{KSS25AnisotropicII}.

\begin{proof}[Proof of \Cref{thm:2}] 
Suppose, to the contrary, the medium does not scatter for some incident wave $u^{\rm inc}$ satisfying either $u^{\rm inc}(x_{0})\neq 0$ or $\nabla u^{\rm inc}(x_{0})\neq 0$. The first two steps are the same as in the proof of the two-dimensional case (\Cref{thm:1}): 
\begin{itemize}
\item First, one obtains the Lipschitz regularity of $u^{\rm sc}$. 
\item Next, one sees that any blowup limit $v$ satisfies (with notation as in \Cref{lem:high-dimen-poly}) 
\begin{equation*}
\mL^{\calC_{0}}v = H\chi_{\mathfrak{C}} \quad \text{in $B_{1}\subset\mR^{d}$} 
\end{equation*}
where $H\not\equiv 0$, and that, for each $i$, the component $H_{i}$ is either identically zero or a homogeneous harmonic polynomial of degree $m$, and $v=0$ in $B_{1}\setminus \mathfrak{C}$. We recall that the $C_{\rm loc}^{1}$ regularity of $v$ is a consequence of the Calder{\'o}n-Zygmund-type estimates for elliptic systems \cite[Theorem~7.3]{GM12EllipticSystems}.
\end{itemize}
In view of \Cref{rem:rotational-symmetry}, assume $\theta_{+}=0$. We now apply Federer's dimension reduction argument. Fix a point $e=e_{d}\in\tilde{\Gamma}_{+}\cap\tilde{\Gamma}_{-}$. Then, for each $i$, the component $H_{i}$ is either identically zero or has a zero at $e$ of order $0 \le k_i\le m$. Let $k = \min\{ k_i \}$. By a direct application of \Cref{lem:Lipschitz}, 
\begin{equation*}
\abs{v(z+e)}+\abs{z}\abs{\nabla v(z+e)} \le C\abs{z}^{k+2}, 
\end{equation*}
which naturally motivates the scaling function $v_{r}:=v(rz+e)/r^{k+2}$. Next, the blowup limit $w$ of $v_{r}$ satisfies \eqref{eq:blowup-in-ball} as in \Cref{lem:high-dimen-poly}, and vanishes in $B_{1}\setminus \mathfrak{C}$. Therefore, \Cref{lem:high-dimen-poly} implies that $w$ is homogeneous of order $k+2$. Arguing as in \cite[Theorem~1.10]{SS25vanishingcontrast}, we deduce that $w$ is independent of $x_{d}$. 

Repeating this argument successively shows that $w$ depends only on $x_{1}$ and $x_{2}$. Let $\mathfrak{C}_2$ be the projection of $\mathfrak{C}$ to the $(x_1, x_2)$ variables. It is straightforward to check that 
\begin{equation*}
\mu\Delta w_{i} = H_{i} \text{ in $\mathfrak{C}_2$} ,\quad w_{i}|_{\{x_{2}=0,x_1 > 0\}}=\partial_{x_{2}}w_{i}|_{\{x_{2}=0, x_1 > 0\}} = 0
\end{equation*}
for all $i=3,\cdots,d$, and 
\begin{equation*}
\left\{\begin{aligned}
& (\lambda + 2\mu)\partial_{1}^{2}w_{1} + \mu\partial_{2}^{2}w_{1} + (\lambda+\mu)\partial_{1}\partial_{2}w_{2} = H_{1}, \\ 
& \mu\partial_{1}^{2}w_{2} + (\lambda + 2\mu)\partial_{2}^{2}w_{2} + (\lambda+\mu)\partial_{1}\partial_{2}w_{1} = H_{2}, 
\end{aligned}\right. 
\end{equation*}
in $\mathfrak{C}_2$, together with the boundary conditions 
\begin{equation*}
w_{j}|_{\{x_{2}=0, x_1 > 0\}}=\partial_{x_{2}}w_{j}|_{\{x_{2}=0, x_1 > 0\}} = 0 \quad \text{for $j=1,2$.} 
\end{equation*}
If $H_{i}\not\equiv0$ for some $i\ge 3$, then the result follows by adapting the arguments of \cite{SS25vanishingcontrast}. On the other hand, if $H_{i}\not\equiv0$ for some $i\in\{1,2\}$, the remaining steps follow from the proof of \Cref{thm:1}. 
\end{proof}

\section{Proof of the main results in the nonvanishing Bernoulli case} 

Without loss of generality, we may assume that $x_{0}=0$. 
Using \eqref{eq:ui-expansion} and the analyticity of $u^{\rm inc}$, we write 
\begin{equation}
\begin{aligned} 
u^{\rm inc}(x) - u^{\rm inc}(0) &= H(x) + R_{0}(x), \\ 
\nabla\otimes u^{\rm inc}(x) &= \nabla\otimes H(x) + R_{1}(x) ,\\ 
\nabla^{\otimes2}\otimes u^{\rm inc}(x) &= \nabla^{\otimes 2}\otimes H(x) + R_{2}(x) 
\end{aligned} \label{eq:decomposition-incident}
\end{equation}
where $H\not\equiv0$ and each component $H_{i}$ vanishes identically or is a harmonic homogeneous polynomial of order $m \ge 1$ and 
\begin{equation*}
\abs{R_{0}(x)}\le C\abs{x}^{m+1} ,\quad \abs{R_{1}(x)}\le C\abs{x}^{m} ,\quad \abs{R_{2}(x)}\le C\abs{x}^{m-1}. 
\end{equation*}
Recall that $u^{\rm sc}$ satisfies \eqref{eq:nonscattering-Bernuolli}. Using \eqref{eq:decomposition-incident}, the functions $f$ and $g$ defined in \eqref{eq:fg-functions} satisfy 
\begin{equation*}
\abs{f(x)} \le C\abs{x}^{m-2} ,\quad \abs{g(x)}\le C\abs{x}^{m-1}. 
\end{equation*}
Therefore, by a direct application of \Cref{lem:Lipschitz} (with $m+2$ being replaced by $m$) we obtain 
\begin{equation}
\abs{u^{\rm sc}(x)} + \abs{x}\abs{\nabla u^{\rm sc}(x)} \le C\abs{x}^{m}, \label{eq:order-solution-bernoulli}
\end{equation}
which naturally motivates the scaling function $\tilde{u}_{r}:=u^{\rm sc}(rx)/r^{m}$, which also satisfies the estimate \eqref{eq:order-solution-bernoulli}. 

In this setting, we say that $v$ is a blowup limit of $u^{\rm sc}$ of order $m$ at $0$ if there is a sequence $r_{j}\rightarrow 0$ so that $\tilde{u}_{r_{j}}\rightarrow v$ in $(C^{0,1}(\overline{B_{1}}))^{d}$ weak-$\star$. 

\subsection{The two-dimensional piecewise \texorpdfstring{$C^{1}$}{C1} obstacle case} 

We begin with the case $d=2$. Let $\Gamma_{\pm}$, $\tilde{\Gamma}_{\pm}$ and $\nu_{\pm}$ be as introduced in \Cref{subsec:2D-case}. Any blowup limit $v\in (C^{0,1}(\overline{B_{1}}))^{2}$ of order $m$ of $u^{\rm sc}$ solves the equation 
\begin{equation*}
\mL^{\calC_{0}}v = \mT^{\tilde{\calC}(0)-\calC_{0}}H \mH^{1}\lfloor\tilde{\Gamma}_{\pm} = \sum_{\pm} \left( \nu_{\pm}\cdot(\tilde{\calC}(0)-\calC_{0}):(\nabla\otimes H) \right)\mH^{1}\lfloor\tilde{\Gamma}_{\pm} 
\end{equation*}
which can be equivalently written as 
\begin{equation}
\mL^{\calC_{0}}v=0 \text{ in $\mathfrak{C}$},\quad v|_{\partial\mathfrak{C}}=0 ,\quad \mT^{\calC_{0}}v|_{\partial\mathfrak{C}} = \mT^{\tilde{\calC}(0)-\calC_{0}}H. \label{eq:Bernoulli-equations}
\end{equation}

For simplicity, we restrict our attention to the simplest case $m=1$. 

\begin{lemma}\label{lem:nonvanishing-Bernoulli-m1}
Assume that $H$ is (harmonic) homogeneous of degree $m=1$. Let $0 < \theta_{0} < 2\pi$ and write 
\begin{equation*}
\mathfrak{C} := \left\{ re^{\bfi\theta} : r>0 , \theta \in (0,\theta_{0}) \right\}. 
\end{equation*}
The unique solution $v$ to 
\begin{equation*}
\mL^{\calC_{0}}v\equiv \mu\Delta v + (\lambda+\mu)\nabla(\nabla\cdot v) = 0 \text{ in $\mathfrak{C}$} ,\quad v|_{\theta=0}=0 ,\quad \mT^{\calC_{0}}v|_{\theta=0} = \mT^{\tilde{\calC}(0)-\calC_{0}}H,
\end{equation*}
is given by 
\begin{equation}
v(re^{\bfi\theta}) = r\sin\theta \left(\begin{array}{cc}1/\mu&0\\0&1/(\lambda+2\mu)\end{array}\right)\mT^{\tilde{\calC}(0)-\calC_{0}}H. \label{eq:soln-order1} 
\end{equation}  
\end{lemma}

\begin{proof}
We first observe that every homogeneous solution of degree $1$ to 
\begin{equation*}
\mL^{\calC_{0}}v\equiv \mu\Delta v + (\lambda+\mu)\nabla(\nabla\cdot v) = 0 \text{ in $\mathfrak{C}$} ,\quad v|_{\theta=0}=0  
\end{equation*}
is of the form 
\begin{equation*}
v(x,y) = \bfc y ,\quad \bfc\in\mC^{2}. 
\end{equation*}
Since 
\begin{equation*}
\mT^{\calC_{0}}v|_{\theta=0} = \left(\begin{array}{cc}\mu&0\\0&\lambda+2\mu\end{array}\right)\partial_{y}v|_{y=0} = \left(\begin{array}{cc}\mu&0\\0&\lambda+2\mu\end{array}\right)\bfc, 
\end{equation*}
\eqref{eq:soln-order1} follows immediately. The uniqueness of this solution follows from either the analyticity of solutions or the unique continuation property, as discussed above. 
\end{proof}

As an immediate consequence, it is easy to see that there are no blowup solutions supported in sectors with opening angle different from $\pi$. This result is essential for the characterization of blowup solutions in two dimensions. 

\begin{corollary}\label{cor:nonvanishing-Bernoulli-m1}
Assume that all the assumptions of \Cref{lem:nonvanishing-Bernoulli-m1} are satisfied. 
If the unique solution $v$ satisfies 
\begin{equation*}
v|_{\theta=\theta_{0}} = 0 
\end{equation*}
for some $\theta_{0}\in(0,2\pi)\setminus\{\pi\}$, then $v\equiv0$. 
\end{corollary}

We are now ready to prove our main result. 

\begin{proof}[Proof of \Cref{thm:3}]
Suppose, for the sake of contradiction, that the medium does not scatter for some incident wave $u^{\rm inc}$ satisfying \eqref{eq:nondegenerate-Bernoulli}. By \Cref{cor:nonvanishing-Bernoulli-m1}, the blowup limit $v$ of $u^{\rm sc}$ of order $m=1$ at $x_{0}$ must vanish identically. It then follows from \eqref{eq:Bernoulli-equations} that 
\begin{equation*}
\lim_{x\rightarrow x_{0}}\mT^{\tilde{\calC}-\calC_{0}}u^{\rm inc}(x) = \mT^{\tilde{\calC}(x_{0})-\calC_{0}}H = 0. 
\end{equation*}
This contradicts \eqref{eq:nondegenerate-Bernoulli}, thereby completing the proof. 
\end{proof}

\subsection{The high-dimensional obstacle case} 

We now extend the above argument to the case $d\ge 3$. Any blowup limit $v\in (C^{0,1}(\overline{B_{1}}))^{d}$ of order $m$ of $u^{\rm sc}$ solves the equation 
\begin{equation*}
\mL^{\calC_{0}}v = \mT^{\tilde{\calC}(0)-\calC_{0}}H \mH^{d-1}\lfloor\tilde{\Gamma}_{\pm} = \sum_{\pm} \left( \nu_{\pm}\cdot(\tilde{\calC}(0)-\calC_{0}):(\nabla\otimes H) \right)\mH^{d-1}\lfloor\tilde{\Gamma}_{\pm} 
\end{equation*}
which can be equivalently written as 
\begin{equation}
\mL^{\calC_{0}}v=0 \text{ in $\mathfrak{C}\times\mR^{d-2}$},\quad v|_{\partial\mathfrak{C}\times\mR^{d-2}}=0 ,\quad \mT^{\calC_{0}}v|_{\partial\mathfrak{C}\times\mR^{d-2}} = \mT^{\tilde{\calC}(0)-\calC_{0}}H. \label{eq:Bernoulli-equations-high-dim}
\end{equation}

For simplicity, we restrict our attention to the simplest case $m=1$. 

\begin{lemma}
Assume that $H$ is (harmonic) homogeneous of degree $m=1$. Let $0 < \theta_{0} < 2\pi$ and write 
\begin{equation*}
\mathfrak{C} := \left\{ re^{\bfi\theta} : r>0 , \theta \in (0,\theta_{0}) \right\}. 
\end{equation*}
The unique solution $v$ to 
\begin{equation*}
\mL^{\calC_{0}}v\equiv \mu\Delta v + (\lambda+\mu)\nabla(\nabla\cdot v) = 0 \text{ in $\mathfrak{C}\times\mR^{d-2}$} ,\quad v|_{\theta=0}=0 ,\quad \mT^{\calC_{0}}v|_{\theta=0} = \mT^{\tilde{\calC}(0)-\calC_{0}}H. 
\end{equation*}
is given by 
\begin{equation}
v(x_{1},x_{2},x_{3},\cdots,x_{d}) = \left[ {\rm diag}\left(\frac{1}{\mu},\frac{1}{\lambda+2\mu},\frac{1}{\mu},\cdots,\frac{1}{\mu}\right)\left.\mT^{\tilde{\calC}(0)-\calC_{0}}H\right|_{x_{2}=0} \right] x_{2}. \label{eq:soln-order1-high-dim} 
\end{equation}  
\end{lemma}

The proof of the above lemma is straightforward, and we therefore omit the details. We are now ready to prove the main result.

\begin{proof}[Proof of \Cref{thm:4}] 
The proof follows essentially the same arguments as the proof of \Cref{thm:3}. 
The key observation is that the expression \eqref{eq:soln-order1-high-dim} is independent of $x_{3},\cdots,x_{d}$. Consequently, the main result can be established without invoking Federer's dimension reduction argument.
\end{proof}

\appendix 
\crefalias{section}{appendix}

\section{Preliminaries on elastic scattering theory\label{appendix}}

For the reader's convenience, we summarize here some preliminaries on elastic scattering theory. For further details, we refer to \cite{Cha02elasticITP}, as well as the classical monograph \cite{KGBB79elastic} and related works \cite{BP08elastic,CGK02LSM,CKAGK07Factorization,KW21CharacterizeNonradiating}.

We begin with a consequence of the standard Helmholtz decomposition, see e.g.\ \cite[Theorem~III.2.2.5, page~123]{KGBB79elastic} or \cite[Lemma~2.1]{KW21CharacterizeNonradiating}. 
In any dimension $d \geq 2$, we define the curl of a vector field $\bfu$ as the skew-symmetric matrix 
\begin{equation*}
(\curl\bfu)_{ij} = \frac{1}{\sqrt{2}}(\partial_{i}u_{j}-\partial_{j}u_{i}). 
\end{equation*}

\begin{lemma}\label{lem:decompose}
Let $\Omega\subset\mR^{d}$ be an open set with $d\ge 2$, and let $\mu>0$ and $\lambda+2\mu>0$. If $u$ is a (smooth) solution to 
\begin{equation}
(\mL^{\lambda,\mu}+\kappa^{2})u=0 \quad \text{in $\Omega$,} \label{eq:Kupradze1}
\end{equation}
for some $\kappa > 0$, then $u$ can be decomposed into compression and shear components $u=u^{(p)}+u^{(s)}$, where 
\begin{equation*}
\left\{\begin{aligned}
& (\Delta + \kappa_{p}^{2})u^{(p)} = 0 ,&& \curl\, u^{(p)} = 0 \\ 
& (\Delta + \kappa_{s}^{2})u^{(s)} = 0 ,&& \nabla\cdot u^{(s)} = 0
\end{aligned}\right.
\end{equation*}
in $\Omega$, where $\kappa_{p}^{2}=\kappa^{2}/(\lambda+2\mu)$ and $\kappa_{s}^{2} = \kappa^{2}/\mu$. 
\end{lemma}

\begin{remark*}
The Helmholtz decomposition states that every vector field $v\in(L^{q}(\Omega))^{d}$, $1<q<\infty$, admits a decomposition $v=v^{(p)}+v^{(s)}$ where $v^{(p)},v^{(s)}\in (L^{q}(\Omega))^{d}$ satisfies $\curl\, v^{(p)}=0$ and $\nabla\cdot v^{(s)} = 0$, see, for example, \cite{FKS07Helmholtz} and the references therein. However, the Helmholtz decomposition generally fails when $\Omega\subset\mR^{2}$ is an infinite cone with `smooth vertex' at the origin and of opening angle larger than $\pi$, see the counterexamples in \cite{Bogovskii1986Helmholtz,MB1986Helmholtz}. The Hodge theorem gives similar decomposition results for differential $k$-forms on oriented compact $d$-dimensional smooth manifolds, see, for example, \cite{Malakhaltsev2008DeRham} and the references therein. 
\end{remark*}

\begin{proof}[Proof of \Cref{lem:decompose}]
First, taking the divergence of \eqref{eq:Kupradze1}, we find that the function $\psi:=\kappa_{p}^{-2}\nabla\cdot u$ satisfies 
\begin{equation*}
(\Delta + \kappa_{p}^{2}) \psi=0 \quad \text{in $\Omega$.}
\end{equation*}
Define $u^{(p)}:=-\nabla\psi$ and $u^{(s)}:=u-u^{(p)}$. 
Then $\curl\, u^{(p)}=0$, and 
\begin{equation*}
\nabla\cdot u^{(s)} = \nabla\cdot u + \Delta\psi = \kappa_{p}^{2}\psi + \Delta\psi = 0 \quad \text{in $\Omega$.} 
\end{equation*}
Finally, we compute 
\begin{equation*}
\begin{aligned}
& \Delta u^{(s)} = \Delta u - \Delta u^{(p)} = -\frac{\lambda+\mu}{\mu}\nabla(\nabla\cdot u) - \kappa_{s}^{2}u +\kappa_{p}^{2} u^{(p)} \\
&\quad = -\frac{\lambda+\mu}{\mu}\kappa_{p}^{2}\nabla\psi - \kappa_{p}^{2} \nabla\psi - \kappa_{s}^{2}u = -\frac{\lambda+2\mu}{\mu}\kappa_{p}^{2}\nabla\psi - \kappa_{s}^{2}u \\
&\quad = \kappa_{s}^{2} u^{(p)} - \kappa_{s}^{2}u = -\kappa_{s}^{2}u^{(s)} \quad \text{in $\Omega$}, 
\end{aligned}
\end{equation*}
which completes the proof. 
\end{proof}

\begin{remark}[curl-curl identity and the uniqueness of the Helmholtz decomposition]\label{rem:Helmholtz-decomposition}
We define the formal transpose of the curl operator by 
\begin{equation*}
(\curl^{\intercal} A)_{i} := \frac{1}{\sqrt{2}}\sum_{j}\partial_{j}(A_{ij}-A_{ji})
\end{equation*}
for any matrix field $A=(A_{ij})$. With these definitions, the following curl-curl identity holds: 
\begin{equation}
\Delta\bfu = \nabla(\nabla\cdot\bfu) - \curl^{\intercal}\curl\,\bfu \label{eq:curl-curl}
\end{equation}
see, for instance, \cite[Section~3]{IKS23UCPMRT} for a generalization to symmetric tensors via Saint-Venant operators. See also \cite{KL19LandisNS,MWZ20LandisNS} for an application to the stationary Navier-Stokes equations, and \cite{FW21magnetohydrodynamic} for an application to the stationary magnetohydrodynamic equations. As an immediate consequence, the decomposition in \Cref{lem:decompose} is unique.
\end{remark}

We are now ready to give a precise definition of the Kupradze radiation condition, following \cite[Definition~III.2.2.6, p.~124]{KGBB79elastic}: 

\begin{definition} \label{def:Kupradze}
Let $R>0$, $\kappa>0$ and assume $\mu>0$ and $\lambda+2\mu>0$. A solution $u$ to \eqref{eq:Kupradze1} in $\Omega=\mR^{d}\setminus B_{R}$ is said to satisfy the \emph{Kupradze radiation condition at infinity} if 
\begin{equation*}
\begin{aligned}
& \lim_{\abs{x}\rightarrow\infty} u^{(p)}(x)=0 , && \lim_{\abs{x}\rightarrow\infty} \abs{x}^{\frac{d-1}{2}}\left( \partial_{\abs{x}}u^{(p)}(x) - \bfi \kappa_{p} u^{(p)}(x) \right) = 0, \\ 
& \lim_{\abs{x}\rightarrow\infty} u^{(s)}(x)=0 , && \lim_{\abs{x}\rightarrow\infty} \abs{x}^{\frac{d-1}{2}}\left( \partial_{\abs{x}}u^{(s)}(x) - \bfi \kappa_{s} u^{(s)}(x) \right) = 0, 
\end{aligned}
\end{equation*}
where $\partial_{\abs{x}}=\hat{x}\cdot\nabla$. 
\end{definition}

\begin{remark}[Far-field pattern] 
The far-field patterns of $u^{(p)}$ and $u^{(s)}$ are defined by 
\begin{equation*}
\begin{aligned}
u^{\infty,(p)}(\hat{x}) &:= \lim_{\abs{x}\rightarrow\infty}\gamma_{d,\kappa_{p}}^{-1}\abs{x}^{\frac{d-1}{2}}e^{-\bfi\kappa_{p}\abs{x}}u^{(p)}(x), \\
u^{\infty,(s)}(\hat{x}) &:= \lim_{\abs{x}\rightarrow\infty}\gamma_{d,\kappa_{s}}^{-1}\abs{x}^{\frac{d-1}{2}}e^{-\bfi\kappa_{s}\abs{x}}u^{(s)}(x), 
\end{aligned}
\end{equation*}
for all $\hat{x}=x/\abs{x}\in\calS^{d-1}$, respectively, where we choose $\gamma_{d,\kappa}:=\frac{e^{-\bfi\pi\frac{d-3}{4}}}{2(2\pi)^{\frac{d-1}{2}}}\kappa^{\frac{d-3}{2}}$ with $\kappa=\kappa_{p},\kappa_{s}$ as in \cite[Section~1.2.3]{Yaf10ScatteringAnalyticTheory}, see also \cite[Section~2]{KSS23Minimization}. By the Rellich uniqueness theorem \cite{CK19scattering,Hormander_rellich} and the uniqueness of the Helmholtz decomposition (cf.\ \Cref{rem:Helmholtz-decomposition}), we have that $u=0$ in $\Omega$ if and only if $(u^{\infty,(p)},u^{\infty,(s)})=(0,0)$. 
\end{remark}

\section*{Acknowledgments}

Kow was supported by the National Science and Technology Council of Taiwan, NSTC 112-2115-M-004-004-MY3 and NSTC 115-2628-M-004-001, and by the National Center for Theoretical Sciences of Taiwan. 
Salo was partly supported by the Research Council of Finland (Centre of Excellence in Inverse Modelling and Imaging and FAME Flagship, grants 353091 and 359208). 
Shahgholian was supported by Swedish Research Council (grant no. 2025-03740).

\section*{Declarations}

\noindent {\bf  AI disclosure:} GPT 5.6 was used for verifying certain computations and for proofreading. This article does not contain AI-written text.

\medskip

\noindent {\bf  Data availability statement:} All data needed are contained in the manuscript.

\medskip
\noindent {\bf  Funding and/or Conflicts of interests/Competing interests:} The authors declare that there are no financial, competing or conflict of interests. 

\end{sloppypar}

\bibliographystyle{custom}
\bibliography{ref}
\end{document}